\documentclass[11pt]{article}

\usepackage{amsmath, amsfonts, amssymb, amsthm,mathrsfs}
\usepackage{mathtools}
\usepackage{enumerate}
\usepackage{graphicx}
\usepackage{color}
\usepackage[margin=1in]{geometry}
\usepackage{lineno}
\usepackage[ruled, vlined, algonl]{algorithm2e}
\usepackage{float}

\DeclareMathOperator{\prox}{prox}
\DeclareMathOperator{\env}{env}
\DeclareMathOperator{\argmin}{argmin}

\DeclareMathOperator{\id}{Id}

\DeclareMathOperator{\SPF}{SPF}

\newcommand{\x}[1]{x^{(#1)}}
\newcommand{\y}[1]{y^{(#1)}}

\newcommand{\uu}[1]{u^{(#1)}}

\newcommand{\ttheta}[1]{\theta^{(#1)}}

\newcommand{\xx}[1]{\tilde{x}^{(#1)}}
\newcommand{\yy}[1]{\tilde{y}^{(#1)}}

\newcommand{\phee}{\varphi}

\newtheorem{theorem}{Theorem}

\newtheorem{lemma}{Lemma}
\newtheorem{definition}{Definition}
\newtheorem{proposition}{Proposition}

\theoremstyle{remark}

\providecommand{\keywords}[1]{{\textit{Keywords:}} #1}
\providecommand{\classification}[1]{Mathematics Subject Classification 2000: #1}

\begin{document}

\title{Algorithmic Versatility of SPF-regularization Methods }
\author{Lixin Shen\thanks{Department of Mathematics, Syracuse University, Syracuse, NY 13244, USA. Email: \texttt{lshen03@syr.edu}} \and {Bruce W. Suter}\thanks{Air Force Research Laboratory, Rome, NY. Email: \texttt{bruce.suter@us.af.mil.}}
\and Erin E. Tripp\thanks{Department of Mathematics, Syracuse University, Syracuse, NY 13244, USA. Email: \texttt{eetripp@syr.edu}}}

\maketitle

\begin{abstract}
\noindent 
Sparsity promoting functions (SPFs) are commonly used in optimization problems to find solutions which are assumed or desired to be sparse in some basis. For example, the $\ell_1$-regularized variation model and the Rudin-Osher-Fatemi total variation (ROF-TV) model are some of the most well-known variational models for signal and image denoising, respectively. However, recent work demonstrates that convexity is not always desirable in sparsity promoting functions. In this paper, we replace convex SPFs with their induced nonconvex SPFs and develop algorithms for the resulting model by exploring the intrinsic structures of the nonconvex SPFs. We also present simulations illustrating the performance of the SPF and the developed algorithms in image denoising.

\vspace{\baselineskip}

\noindent \keywords{Nonlinear optimization; sparsity promoting; variational models; image denoising.}

\vspace{\baselineskip} 

\noindent \classification{68U10, 65T60.}
\end{abstract}


\section{Introduction}
Sparsity is identified as a crucial assumption in various applications ranging from signal processing to machine learning and statistics. The widespread interest in sparsity can be attributed to the fact that (i) sparsity infers intrinsic structures of data and (ii) sparse data is easier to manipulate and interpret. Informally, data in the form of vector or matrix is sparse if it contains few nonzero entries. The natural mathematical measure of sparsity is the so-called ``$\ell_0$-norm", which counts the number of nonzero entries in a vector. In the context of optimization, this measure can be viewed as a penalty on non-sparse solutions, and it is in this context that we call the $\ell_0$-norm a sparsity promoting function (SPF). However, solving $\ell_0$-penalized optimization problems is known to be NP-hard. To overcome this difficulty, $\ell_1$-regularization methods such as least absolute shrinkage and selection operator (LASSO) \cite{Tibshirani:jrss:96} and Dantzig selectors \cite{Candes-Tao:IEEE-TIT:06} have been proposed. This relaxation allows application of the many tools of convex analysis, making the problem numerically tractable, but it also introduces bias by heavily penalized entries with large magnitude. To address this, nonconvex penalties have been proposed to replace the $\ell_1$-penalty, including the $\ell_p$-norm with $0 < p < 1$ \cite{Chen-Shen-Suter:IET:16,Frank-Friedman:Technometrics:1993}, the smoothly clipped absolute deviation penalty \cite{Fan-Li:JASA:01}, the continuous exact $\ell_0$ penalty \cite{Soubies-Blanc-Feraud-Aubert:SIAMOPT:2017}, and the minimax concave penalty (MCP) \cite{Zhang:AS:2010}. There is increasing evidence that supports  the use of nonconvex penalties in many applications, see, for example  \cite{An-Tao:AOR:2005,lethi-etal:Eu-JOR:2015,Yin-Lou-He-Xin:SIAMSC:2015} and the references therein. Like the $\ell_0$-norm, these penalty functions are all widely accepted as SPF, and, as noted in \cite{Fan-Li:JASA:01}, they all share certain essential properties.

Based on these observations, we have attempted to give a formal mathematical definition of SPFs in our recent work \cite{Shen-Suter-Tripp:2019}. Loosely speaking, a function is a SPF if its subdifferential at the origin contains the origin and at least one other element; that is, a SPF has a corner or cusp at the origin. Viewed another way, the subdifferential of the function at the origin is a set which defines a threshold for ``small" entries which are considered noise. In terms of the proximity operator, the proximity operator of the SPF will send all elements under this threshold to the origin. Fortunately, all of the above penalties fit this definition.

In \cite{Shen-Suter-Tripp:2019},  we introduced a family of SPFs each of which is the difference of a convex SPF with its Moreau envelope.
Functions in this family have the desired nonconvexity but enough useful properties to develop efficient algorithms for optimization problems penalized by these SPFs. These functions are non-negative, semiconvex, and a special case of difference of convex functions with one term having a Lipschitz continuous gradient. Due to these properties, we refer to these functions as structured SPFs. As an example, the MCP is a particular instance of this construction. Many other examples and interesting properties of the structured SPFs can be found in \cite{Shen-Suter-Tripp:2019}.

The goal of this paper is to demonstrate the applicability of structured SPFs to a variety of optimization models. To illustrate these ideas, we consider the regularized least squares model:
\begin{equation}\label{eq:MainGoal}
\argmin \left\{ \frac{1}{2\lambda} \|x-z\|^2 + \left( \Phi \circ B \right)(x) : x \in C \right\},
\end{equation}
where $C$ is a closed convex subset of $\mathbb{R}^d$, $\lambda$ is a regularization parameter, $z \in \mathbb{R}^d$, $B \in \mathbb{R}^{n \times d}$, and $\Phi$ is a sparsity promoting function on $\mathbb{R}^n$. We note that all of the discussion and results below hold true if the quadratic term is replaced by a differentiable strongly convex function. We simply choose this model as our prototype because of its simplicity as well as its applicability. Problems of interest in the context of image/signal processing at large can be formulated as finding a solution to \eqref{eq:MainGoal}. For example, if $z$ is an image corrupted by Gaussian noise, $\Phi \circ B$ is a composition of the $\ell_2$-norm with the two-dimensional first order difference operator, model \eqref{eq:MainGoal} reduces to the well known Rudin-Osher-Fatemi total variation (ROF-TV) model. If $\Phi$ is the $\ell_1$-norm and $B$ is formulated from a tight framelet, then the resulting model \eqref{eq:MainGoal} was discussed in \cite{shen:IEEEIP:06}.

We propose replacing convex $\Phi$ with the structured SPFs. The flexibility provided by these functions allows us to approach \eqref{eq:MainGoal} from several perspectives and to make use of algorithmic advances in convex, difference of convex, and nonconvex optimization.
In each case, we are able to see how the structures of the SPF plays out in the algorithms and to what effect.  More precisely, three different algorithms for model~\eqref{eq:MainGoal} will be proposed by fully employing the various properties of $\Phi$. The first algorithm explores the semiconvexity property of $\Phi$ to identify the objective function of \eqref{eq:MainGoal} in a form which can be optimized by the primal-dual splitting algorithm in \cite{Condat:JOTA:2013}. The second algorithm is based on the natural difference of convex form of $\Phi$, by its design, so that the difference of convex algorithm (DCA), e.g., in \cite{LeThi:MP:2018,lethi-etal:Eu-JOR:2015,Tao-An:SIAMOPT:1998}, can be applied directly. As shown in our previous work \cite{Shen-Suter-Tripp:2019}, the proximity operators of many constructed structured SPFs have explicit expressions available, but, not utilizing it in the development of the above two algorithms. The third algorithm makes use of the explicit form of the proximity operator of $\Phi$. The convergence analysis of these algorithms and their applications in image denoising will be provided. Numerical results demonstrate increased accuracy without additional computational time in many instances.

The rest of the paper is organized in the following manner. In the next section we recall some necessary background in optimization, briefly review the definition of structured SPFs, and point out some properties of these functions that will be explored in the development of algorithms suitable for model~\eqref{eq:MainGoal}. We also give examples that can be used in  the model~\eqref{eq:MainGoal} for image denoising application. In Section~\ref{sec:Algorithms}, the properties of structured SPFs are used to develop efficient algorithms for  model~\eqref{eq:MainGoal}. Numerical experiments are conducted in Section~\ref{sec:experiments} to demonstrate the performance of the developed algorithms in image denoising. Our conclusions are drawn in Section~\ref{sec:conclusions}.


\section{Structured Sparsity Promoting Functions}\label{sec:SSPF}
In this section, we define precisely what we mean by sparsity promoting functions (SPFs) as well as the family of structured SPFs introduced in our recent paper \cite{Shen-Suter-Tripp:2019}. The relevant results to this paper are presented, and two examples of interest are provided in detail.

We begin by introducing our notation. We denote by $\mathbb{R}^d$ the usual $d$-dimensional Euclidean space  equipped with the standard inner product $\langle \cdot, \cdot \rangle$ and the induced Euclidean norm $\|\cdot\|$. A function $p$ defined on $\mathbb{R}^d$ with values in $\mathbb{R} \cup \{+\infty\}$ is proper if its domain $\mathrm{dom}(p)=\{x\in \mathbb{R}^d: p(x)<+\infty\}$ is nonempty, and $p$ is  lower semicontinuous if its epigraph is a closed set. The set of proper and lower semicontinuous functions on  $\mathbb{R}^d$ to $\mathbb{R} \cup \{+\infty\}$ is denoted by $\Gamma(\mathbb{R}^d)$. The set of proper, convex, and lower semicontinuous functions on  $\mathbb{R}^d$ to $\mathbb{R} \cup \{+\infty\}$ is denoted by $\Gamma_0(\mathbb{R}^d)$.

The subdifferential and proximity operator of a  lower semicontinuous function are two important concepts in nonlinear optimization. We review some aspects of these concepts that are needed in this paper.  Recall that the Fr\'{e}chet subdifferential of a function $p: \mathbb{R}^d \rightarrow \mathbb{R} \cup \{+\infty\}$ at $z \in \mathbb{R}^d$, denoted by $\partial p (z)$, is defined as
$$
\partial p(z):=\left\{t\in\mathbb{R}^d: \liminf_{u\rightarrow z} \frac{p(u)-p(z)-\langle t, u-z\rangle}{\|u-z\|}\ge 0\right\}.
$$
The set $\partial p(z)$ is closed and convex. If $\partial p(z) \neq \emptyset$, we say that $p$ is Fr\'{e}chet subdifferentiable at $z$. If $p$ is convex, then
$\partial p(z):=\left\{t\in\mathbb{R}^d: p(u)-p(z) \ge \langle t, u-z\rangle, \; u\in\mathbb{R}^d\right\}$.
If $p$ is Fr\'{e}chet differentiable at $z$ with a derivative, then $\partial p(z)=\{\nabla p (z)\}$.

We further review some useful simple calculus results for Fr\'{e}chet subdifferentials. If a function $p: \mathbb{R}^d \rightarrow \mathbb{R} \cup \{+\infty\}$ attains its local minimum at $z \in \mathbb{R^d}$, then $0 \in \partial p(z)$ and the point $z$ is called a critical point of $p$. For any $\alpha>0$, it holds that $\partial (\alpha p)(z)=\alpha \partial p(z)$. For any functions $p_1: \mathbb{R}^d \rightarrow \mathbb{R} \cup \{+\infty\}$  and $p_2: \mathbb{R}^d \rightarrow \mathbb{R} \cup \{+\infty\}$ Fr\'{e}chet subdifferentiable at $z$, then $p_1+p_2$ is Fr\'{e}chet subdifferential at $z$ and $\partial(p_1+p_2)(z) \subseteq \partial p_1(z)+\partial p_2(z)$. If one of the above functions is Fr\'{e}chet differentiable at $z$, say $p_1$, then $\partial(p_1+p_2)(z) = \nabla p_1(z)+\partial p_2(z)$.

The proximity operator was introduced by Moreau in \cite{moreau:RASPS:62,moreau:BSMF:65}. For a function $p \in \Gamma(\mathbb{R}^d)$, the proximity operator of $p$ at $z \in  \mathbb{R}^d$ with index $\alpha$ is defined by
$$
\mathrm{prox}_{\alpha p} (z) := \mathrm{arg}\min \left\{p(w)+\frac{1}{2\alpha} \|w-z\|^2: w\in\mathbb{R}^d\right\}.
$$
The proximity operator of $p$ is a set-valued operator from $\mathbb{R}^d \rightarrow 2^{\mathbb{R}^d}$, the power set of $\mathbb{R}^d$. Clearly, for any $w^\star \in \mathrm{prox}_{\alpha p} (z)$,  by the calculus of Fr\'{e}chet subdifferential, we have that
\begin{equation}\label{eq:diff-env}
\frac{1}{\alpha}(z-w^\star) \in \partial p (w^\star).
\end{equation}
The Moreau envelope of  $p$ at $z \in  \mathbb{R}^d$ with index $\alpha$, denoted by $\mathrm{env}_{\alpha} p(z)$ is closed related to the proximity operator $\mathrm{prox}_{p} (z)$. That is,
$$
\mathrm{env}_{\alpha}p (z) := p(w^\star)+\frac{1}{2\alpha} \|w^\star-z\|^2,
$$
where $w^\star$ is in $\mathrm{prox}_{\alpha p} (z)$. If $p$ is convex, then the proximity operator of $p$ is a single-valued operator from $\mathbb{R}^d \rightarrow \mathbb{R}^d$. Furthermore,  equation~\eqref{eq:diff-env} becomes
$$
\frac{1}{\alpha}(\id-\mathrm{prox}_{\alpha p}) (z) \in \partial p (\mathrm{prox}_{\alpha p} (z)).
$$
With this preparation, we are now able to give our definition of sparsity promoting functions and describe some of their properties.

\begin{definition}\label{defn:SPF} We say a function $\phee \in \Gamma(\mathbb{R}^d)$ is sparsity promoting if (i) $\phee$ achieves its global minimum of zero at the origin and (ii) there is a nonzero element in $\partial \phee(0)$.  Denote by $\SPF(\mathbb{R}^d)$ the set of sparsity promoting functions on $\mathbb{R}^d$. \end{definition}

The first item of Definition \ref{defn:SPF} ensures that nonzero entries are penalized. The second item describes the necessary sharpness of SPF, and the set $\partial \phee (0)$ defines what is considered small and therefore what should be sent to zero. For example, if $\phee \in \Gamma_0(\mathbb{R}^d)$, then \eqref{eq:diff-env} becomes
\[ \prox_{\alpha \phee}(x) = 0 \iff x \in \alpha \partial \phee(0). \]
We note that all of the penalties discussed above satisfy this definition, as does any norm on $\mathbb{R}^d$.

Now for any convex $\phee \in \SPF(\mathbb{R}^n)$ and any $\alpha > 0$, we define
\begin{equation}\label{defn:falpha}
\phee_\alpha = \phee - \env_\alpha \phee.
\end{equation}
By construction, $\phee_\alpha$ is a nonnegative difference of convex functions. We summarize relevant properties of $\phee_\alpha$ below. Based on these properties, we refer to these functions as structured SPF's.

\begin{lemma}\label{lemma:properties}
Given a convex function $\phee \in \SPF(\mathbb{R}^d)$ and $\alpha > 0$, the function $\phee_\alpha$ defined by \eqref{defn:falpha} has the following properties:
\begin{enumerate}[(i)]
    \item $\phee_\alpha \in \SPF(\mathbb{R}^d)$ with $\partial \phee_\alpha(0) = \partial \phee(0)$;
    \item $\phee_\alpha$ is $\frac{1}{\alpha}$-semiconvex, i.e. $\phee_\alpha + \frac{1}{2\alpha}\| \cdot \|^2$ is convex;
    \item given $B \in \mathbb{R}^{n\times d}$, $\phee_\alpha \circ B$ is $\frac{\|B\|^2}{\alpha}$-semiconvex.
\end{enumerate}
\end{lemma}
\begin{proof}\ \ The proofs of items (i) and (ii) can be found in \cite{Shen-Suter-Tripp:2019}. We now turn to prove item (iii). Define $\psi=\phee_\alpha + \frac{1}{2\alpha}\| \cdot \|^2$. Then, for any $x \in \mathbb{R}^n$
\begin{equation}\label{temp:1}
\phee_\alpha \circ B (x)+\frac{\|B\|^2}{2\alpha}\|x\|^2 = \psi (Bx) + \frac{\|B\|^2}{2\alpha}\|x\|^2- \frac{1}{2\alpha}\|Bx\|^2.
\end{equation}
By item (ii), $\psi \circ B$  is convex. Note that $\|B\|^2 \|x\|^2-\|Bx\|^2=x^\top (\|B\|^2 \id - B^\top B) x \ge 0$ for all $x \in \mathbb{R}^d$, so it is convex. Hence,  $\phee_\alpha \circ B+\frac{\|B\|^2}{2\alpha}\|\cdot\|^2$ is convex, which implies that item (iii) holds.
\end{proof}

One benefit of Definition \ref{defn:SPF} is that it is sufficiently general to encompass many examples. In practice, we often require more of $\phee$ than convexity and can therefore specify further properties of $\phee_\alpha$. Properties such as separability or block-separability are assumed to control the fineness of sparsity enforcement, and convergence analysis may rely on the function being continuous or subanalytic. In each of these cases, $\phee_\alpha$ inherits the given properties. This is evident in the following examples.

\subsection{Example 1: $\phee$ is the absolute value function}


Relying on the separability of the $\ell_1$-norm, we simply let $\phee$ be the absolute value function, that is,  $\phee(x)=|x|$ on $\mathbb{R}$. Clearly, because $\phee$ achieves its minimum at the origin and $\partial \phee(0)=[-1, 1]$, the absolute value function on $\mathbb{R}$ is a sparsity promoting function. The proximity operator and the Moreau envelope of $|\cdot|$ with parameter $\alpha>0$ are
$$
\mathrm{prox}_{\alpha |\cdot|} (x)=\mathrm{sgn}(x) \max\{0, |x|-\alpha\} \quad \mbox{and} \quad
\mathrm{env}_{\alpha} |\cdot| (x) = \left\{
                                \begin{array}{ll}
                                  \frac{1}{2\alpha}x^2, & \hbox{if $|x|\le \alpha$;} \\
                                  |x|-\frac{1}{2}\alpha, & \hbox{otherwise,}
                                \end{array}
                              \right.
$$
respectively. It is well known that $\mathrm{prox}_{\alpha |\cdot|}$ is called the soft thresholding operator in wavelet literature \cite{Donoho:ieeeIT:06} and $\env_{\alpha} |\cdot|$ is Huber's function in robust statistics \cite{Huber:09}. We note that for $x \in \mathbb{R}^d$, $\prox_{\alpha \|\cdot\|_1}(x) = \prox_{\alpha |\cdot|}(x_1) \times \cdots \times \prox_{\alpha |\cdot|}(x_d)$ and $\env_\alpha \|\cdot\|_1(x) = \sum_{i = 1}^d \env_\alpha |\cdot| (x_i)$.

Figure~\ref{figure:ex1-f-envf}(a) depicts the graphs of $|\cdot|$ (solid line) and its Moreau envelope (dotted line) while  Figure~\ref{figure:ex1-f-envf}(b) shows the graph of its the proximity operator.

\begin{figure}[h]\centering
\begin{tabular}{ccc}
\includegraphics[scale=0.32]{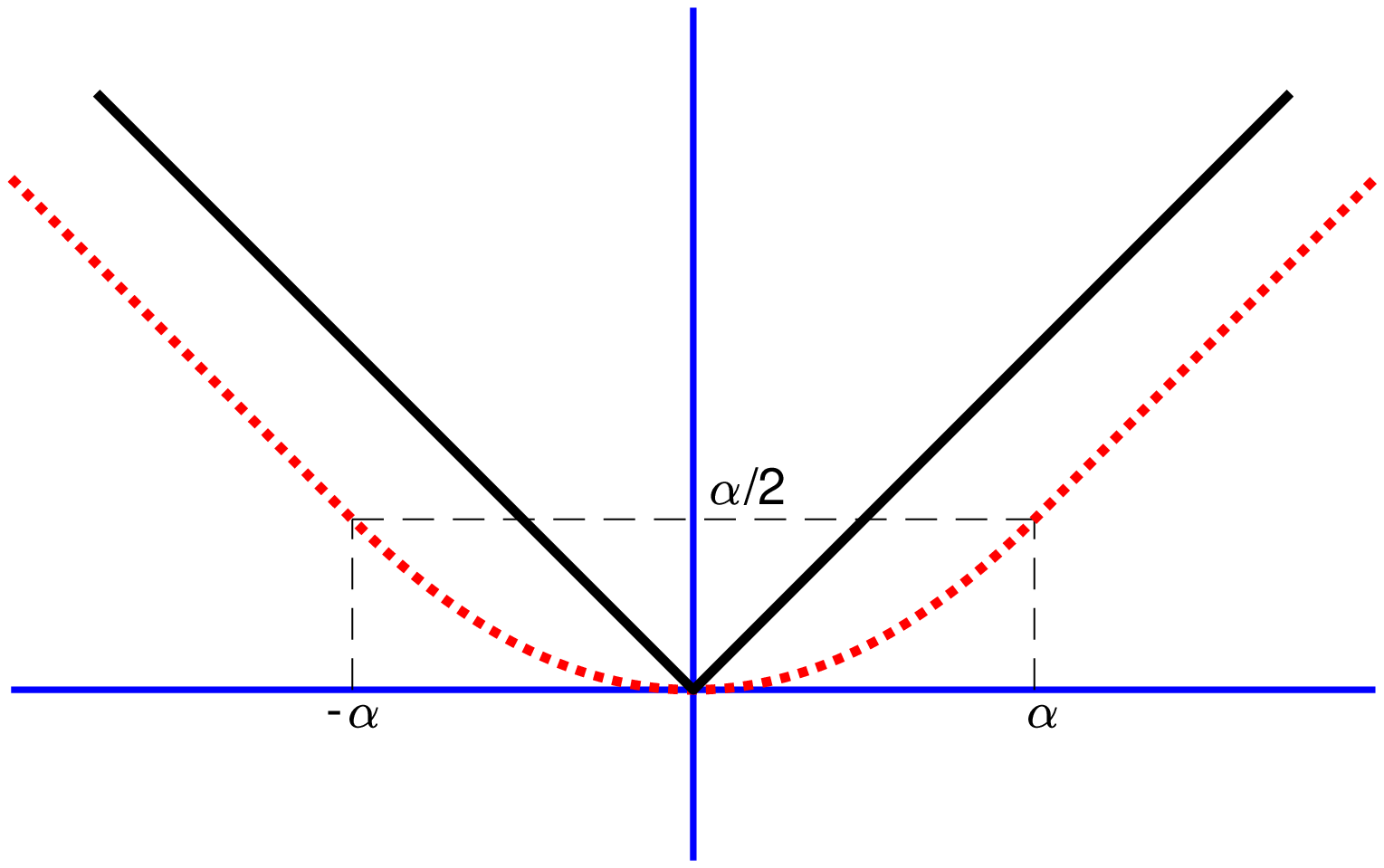}&
\includegraphics[scale=0.32]{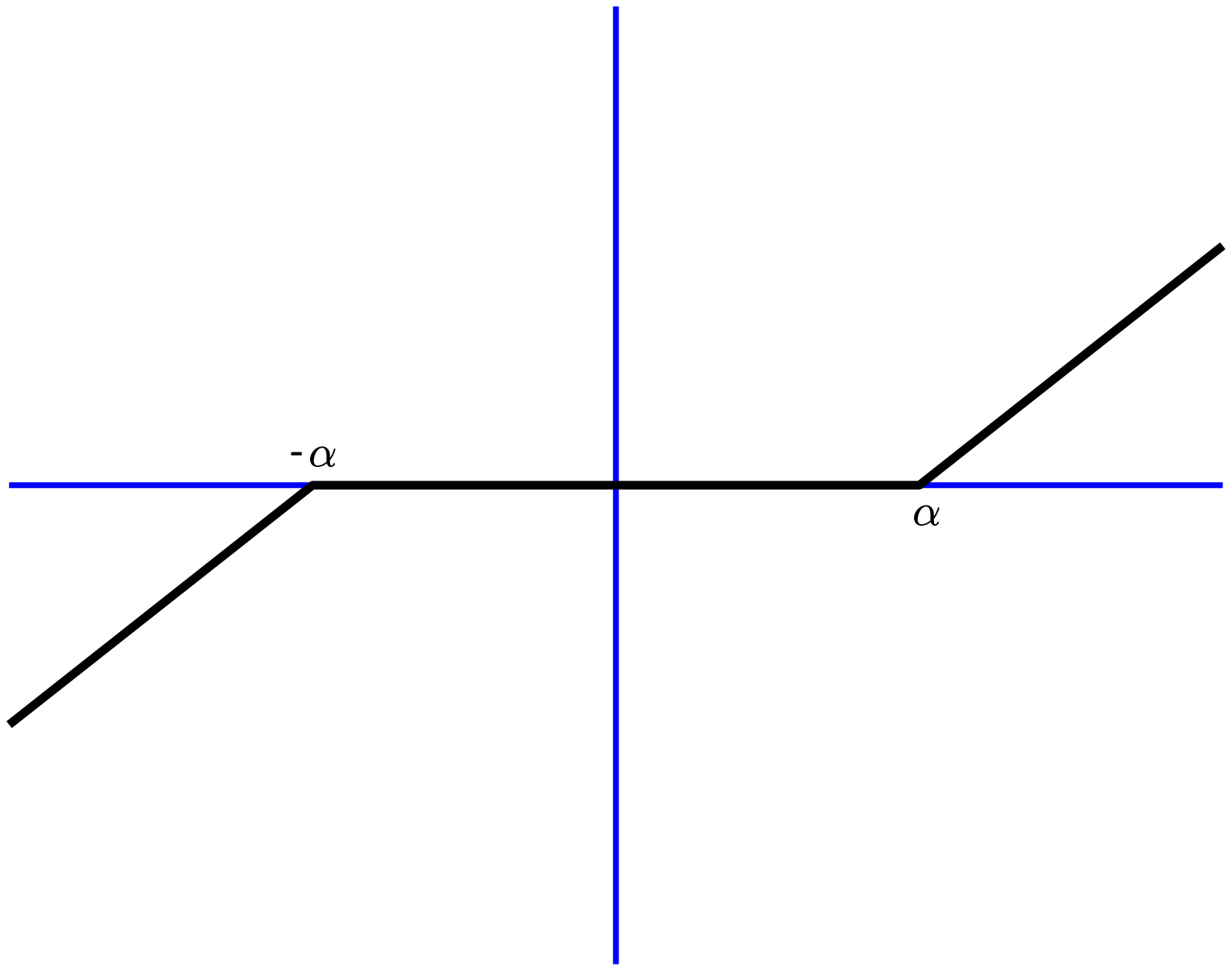}&
\includegraphics[scale=0.32]{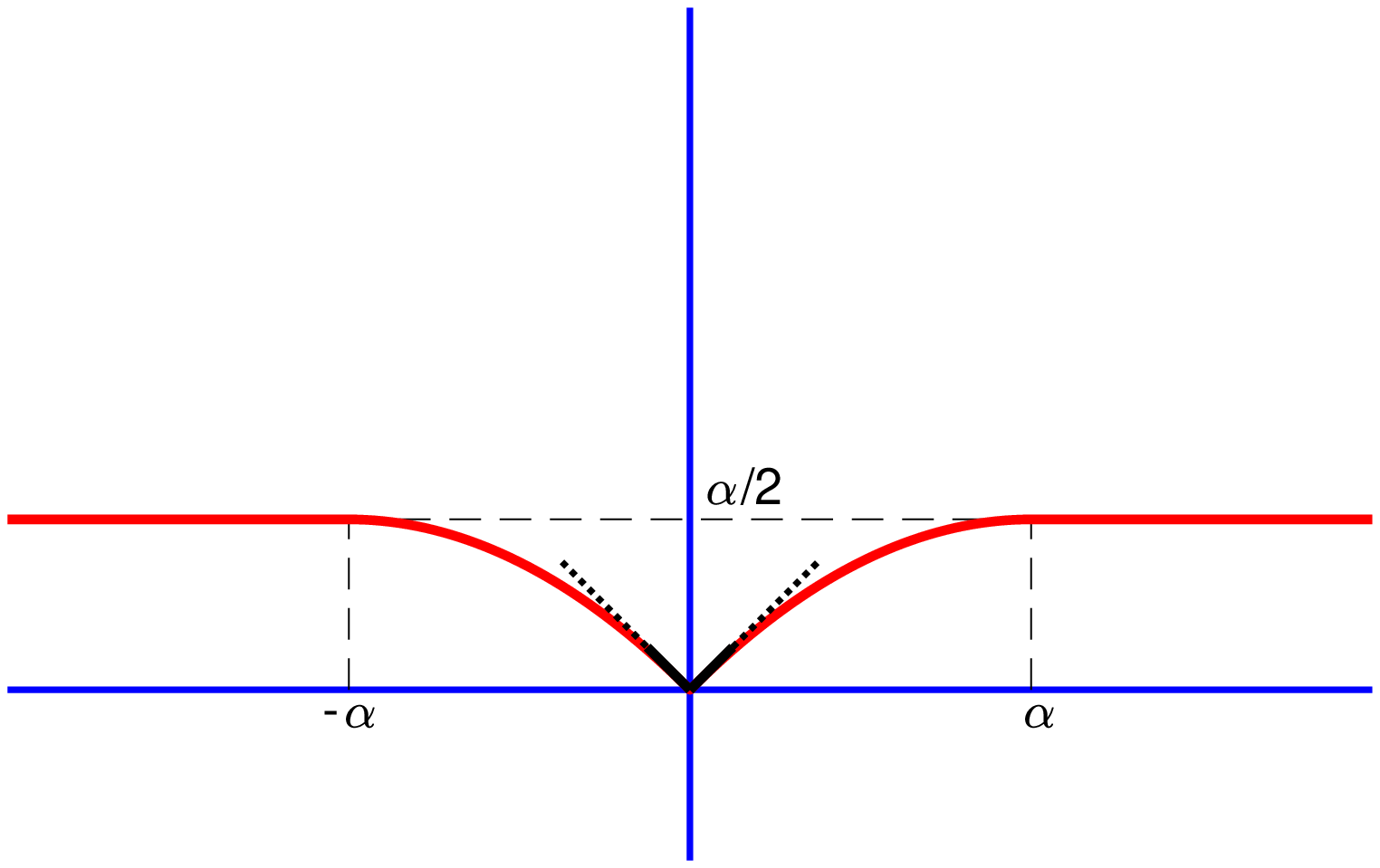}\\
(a) &(b)&(c)
\end{tabular}
\caption{Let $\phee=|\cdot|$ be the absolute value function. (a) The graphs of $\phee$ (solid), $\mathrm{env}_{\alpha} \phee$ (dotted); (b) The typical shape of $\prox_{\alpha \phee}$;  and (c) the graph of $\phee_\alpha = \phee(x) - \mathrm{env}_{\alpha} \phee(x)$. Near the origin $\phee_\alpha$ retains the structure of $f$, which is emphasized in black (solid-dotted).} \label{figure:ex1-f-envf}
\end{figure}

As defined in \eqref{defn:falpha}, for the absolute value function $\phee$,
$$
\phee_\alpha(x):= |x|-\env_\alpha |\cdot| (x)=\left\{
                                \begin{array}{ll}
                                  |x|-\frac{1}{2\alpha}x^2, & \hbox{if $|x|\le \alpha$;} \\
                                  \frac{1}{2}\alpha, & \hbox{otherwise.}
                                \end{array}
                              \right.
$$
This function $\phee_\alpha$ (see Figure~\ref{figure:ex1-f-envf}(c)) is identical to the minimax convex penalty (MCP) function given in \cite{Zhang:AS:2010}, but motivated from statistic perspective.  It is straightforward to extend this to $\mathbb{R}^d$: $(\|\cdot\|_1)_\alpha(x) = \sum_{i = 1}^d \phee_\alpha(x_i)$. The expression of $\mathrm{prox}_{\beta \phee_{\alpha}}$ depends on the relative values of $\alpha$ and $\beta$, and takes the form as follows (see \cite{Shen-Suter-Tripp:2019}):
\begin{equation}\label{eq:beta-alpha}
\mathrm{prox}_{\beta \phee_{\alpha}}(x)=\left\{
                                        \begin{array}{ll}
                                          \frac{\alpha}{\alpha-\beta}(|x|-\beta)\cdot \mathrm{sgn}(x) \cdot \max\{|x|-\beta,0\} \chi_{\{|x|\le \alpha\}}+\{x\}\chi_{\{|x|> \alpha\}}, & \hbox{if $\beta<\alpha$;} \\
                                          \{0\}\chi_{\{|x|<\alpha\}}+\mathrm{sgn}(x) \cdot[0, \alpha]\chi_{\{|x|= \alpha\}}+\{x\} \chi_{\{|x| >\alpha\}}, & \hbox{if $\beta=\alpha$;} \\
                                         \{0\}\chi_{\{|x|<\alpha\}}+\mathrm{sgn}(x) \cdot\{0, \sqrt{\alpha\beta}\}\chi_{\{|x|= \sqrt{\alpha\beta}\}}+\{x\} \chi_{\{|x| >\sqrt{\alpha\beta}\}} , & \hbox{if $\beta>\alpha$.}
                                        \end{array}
                                      \right.
\end{equation}
Here $\chi_S$ has value $1$ at points of the set $S$, and $0$ at points of $\mathbb{R}\setminus S$. The graphs of  $\mathrm{prox}_{\beta \phee_{\alpha}}$ for different values of $\alpha$ and $\beta$ are plotted in Figure~\ref{fig:prox-|x|}.
\begin{figure}[h]
\centering
\begin{tabular}{ccc}
\includegraphics[scale=0.4]{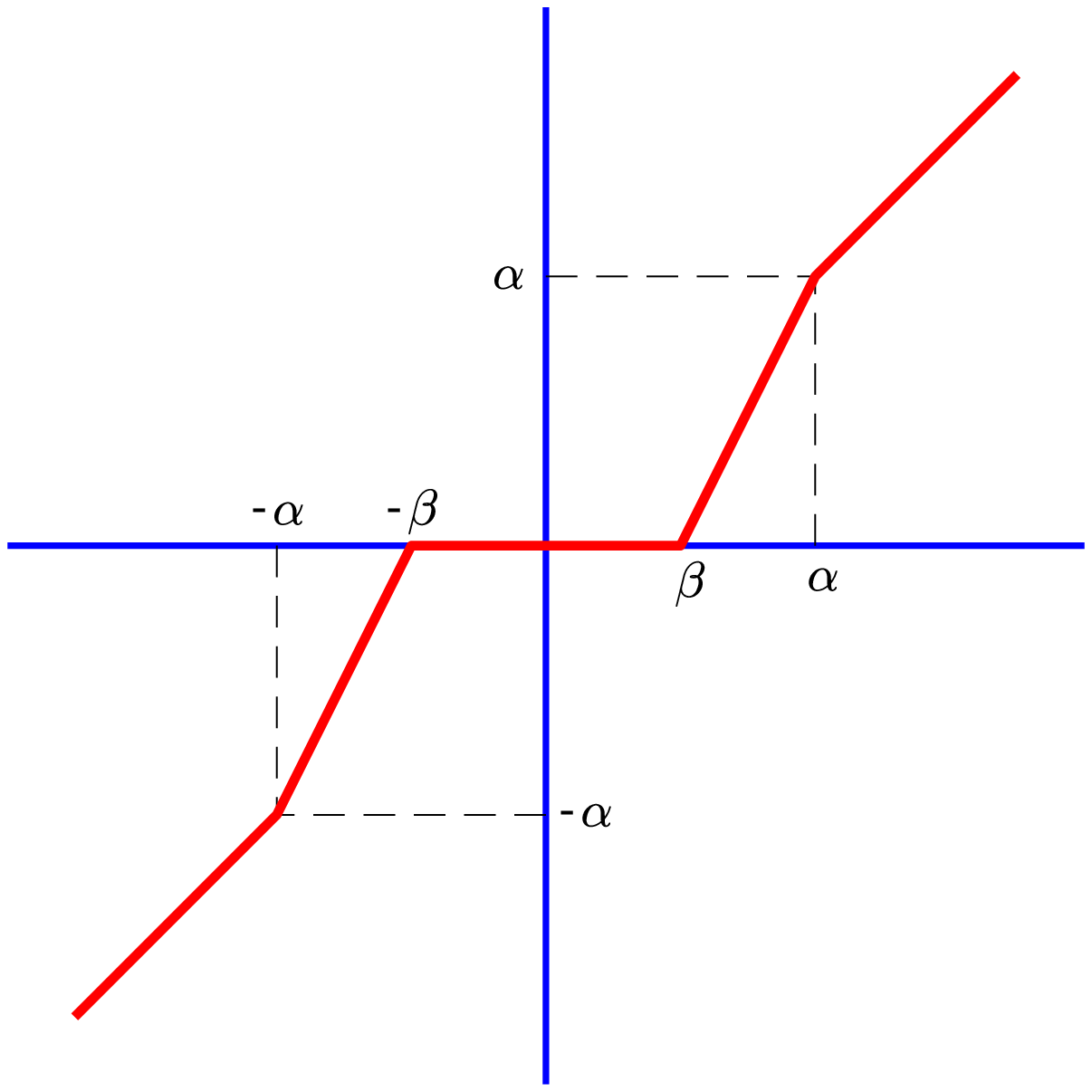} & \includegraphics[scale=0.4]{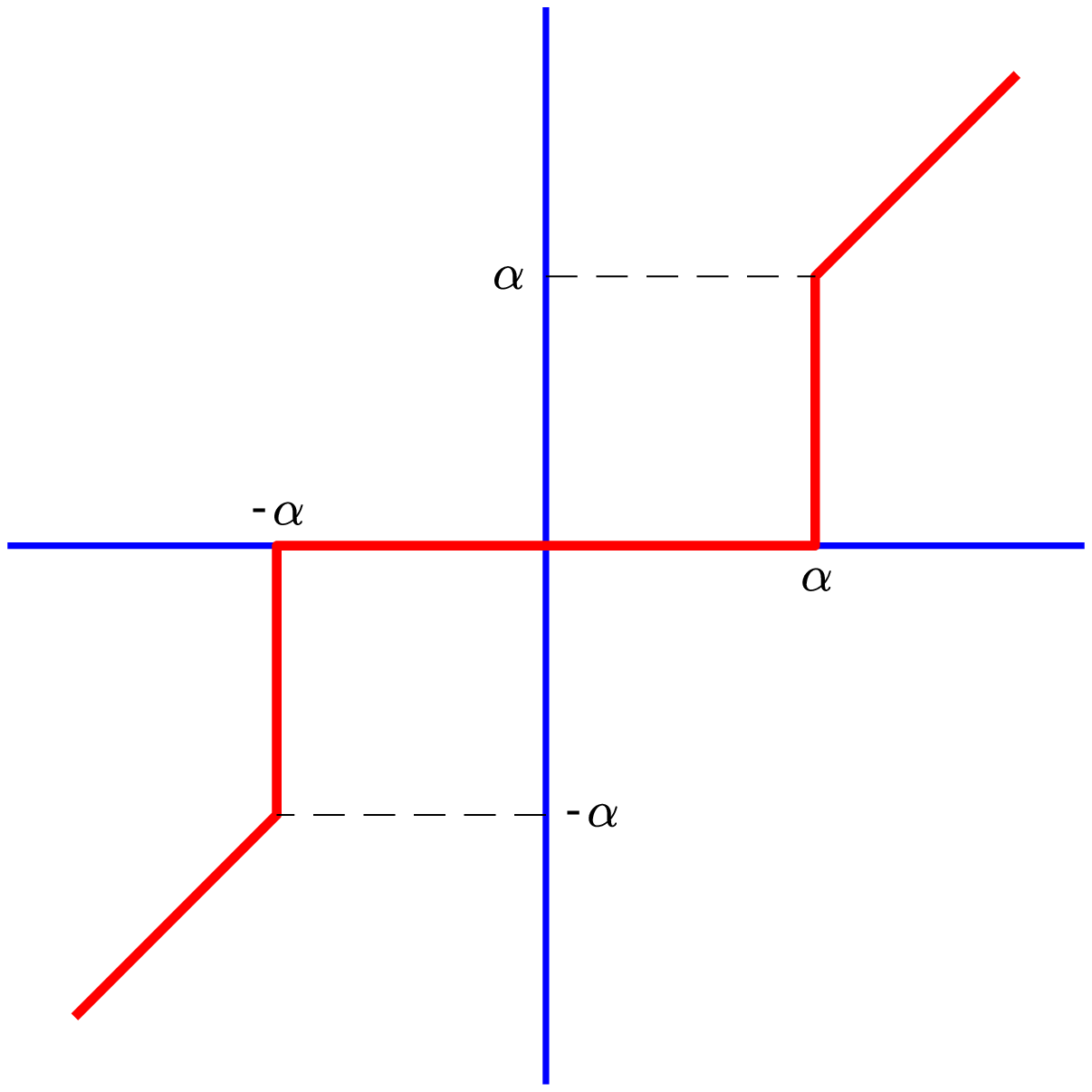} & \includegraphics[scale=0.4]{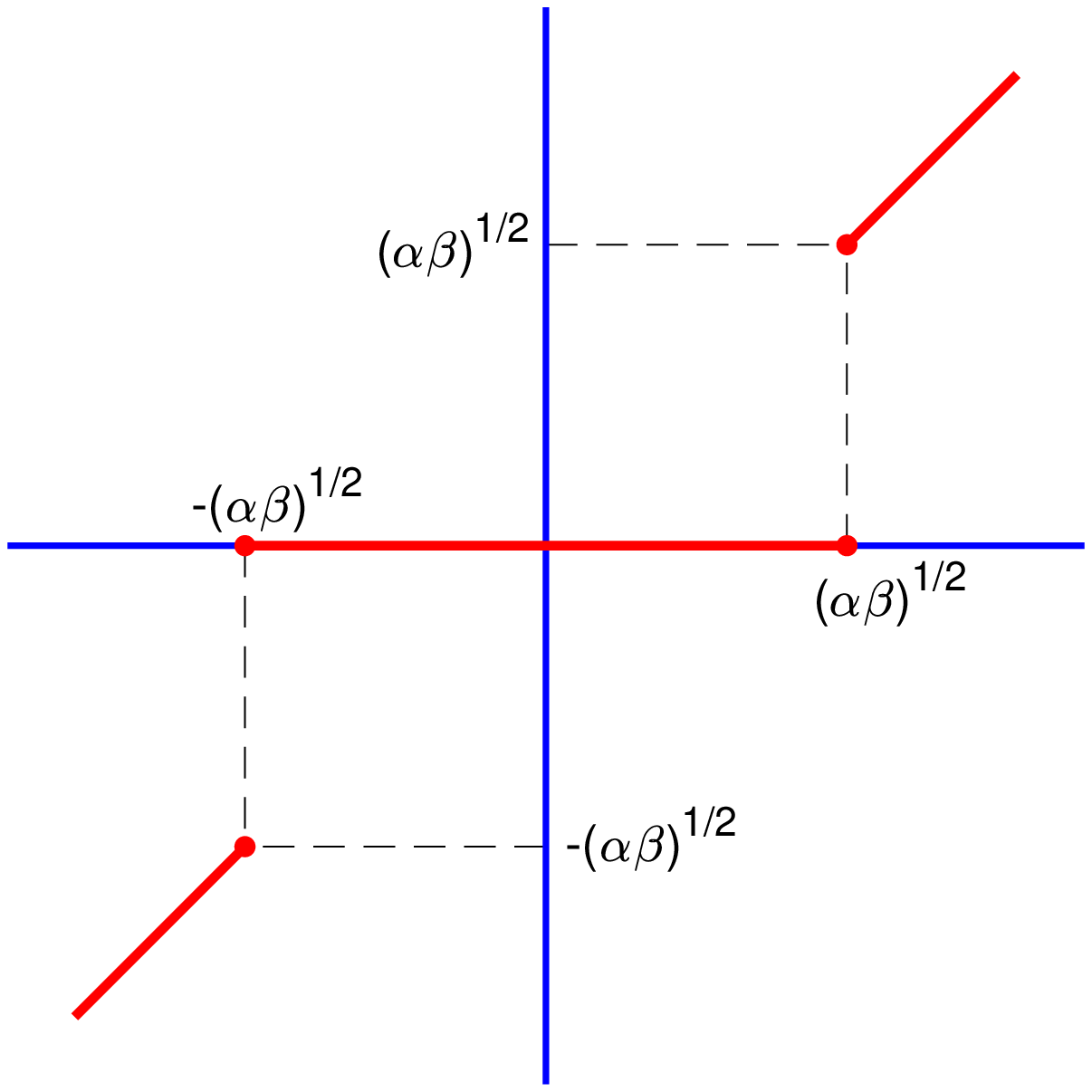}\\
(a) & (b) & (c)
\end{tabular}
\caption{Typical shapes of the proximity operator of $| \cdot |_\alpha$ for (a) $\beta < \alpha$, (b) $\beta = \alpha$, (c) $\beta > \alpha$. The sparsity threshold and the thresholding behavior depend on the relationship between $\alpha$ and $\beta$.}
\label{fig:prox-|x|}
\end{figure}

\subsection{Example 2:  $\phee$ is a compositional norm}

The example here is motivated from the total variation that will be defined in Section~\ref{sec:experiments}.  To define this function, let the disjoint sets $\omega_j$, $j=1,\dots,J$ be the partition of the set $\{1, 2, \ldots, d\}$, that is $\cup_{j=1}^J \omega_j=\{1, 2, \ldots, d\}$; and  let $I_{\omega_j}$ be the $\#{\omega_j} \times d$ matrix formed by those rows of the $d \times d$ identity matrix with indices in ${\omega_j}$. Since $I_{\omega_j}x$ for $x\in \mathbb{R}^d$ is the vector whose entries are from those of $x$ with indices in $\omega_j$, we call $I_{\omega_j}$ extraction matrix.   With these preparation, in the second example, we will consider the following function: for $x\in \mathbb{R}^d$
\begin{equation}\label{def:exampe2}
\phee(x)=\sum_{j=1}^J \|I_{\omega_j} x\|.
\end{equation}
It is not difficult to show that $\phee$ in \eqref{def:exampe2} is a norm of $\mathbb{R}^d$ (associated with the given partition). In \cite{Bradley-Bagnell:2009:CC}, $\phee$ in \eqref{def:exampe2} is referred to as a compositional norm since it is a norm composed of norms over disjoint sets of variables.

For this example, we will show that $\phee$ in \eqref{def:exampe2} is a sparsity promoting function and $\phee_\alpha$ can be presented in terms of $|\cdot|_\alpha$ from example 1. Indeed, the following result says that $\phee$ in \eqref{def:exampe2} is a sparsity promoting function.
\begin{proposition}\label{prop:properties-ex2}
Let $\phee$ in \eqref{def:exampe2} be a compositional norm on $\mathbb{R}^d$, then $\phee \in \SPF(\mathbb{R}^d)$.
\end{proposition}
\begin{proof} Clearly, $\phee(0)=0$ and $\phee$ achieves its global minimum at the origin. Further, we have $\partial \phee(x)=\sum_{j=1}^J  I^\top_{\omega_j} \partial \|\cdot\| (I_{\omega_j} x)$. Since $\partial \|\cdot\| (I_{\omega_j} 0)$ is the unit ball of $\mathbb{R}^{\#{\omega_j}}$, we know that $\partial \phee(0)$ contains nonzero elements, so $\phee$ is a SPF.
\end{proof}

To compute $\phee_\alpha$ and its proximity operator,  we need the following lemma, which can be viewed as an extension of the first example.
\begin{lemma}\label{lemma:L1-L2}
Let $f$ be the $\ell_2$-norm on $\mathbb{R}^d$, that is, $f=\|\cdot\|$. Then, it holds that for any nonzero $x \in \mathbb{R}^d$ and two positive parameters $\alpha$ and $\beta$
$$
\mathrm{prox}_{\beta f}(x) = \mathrm{prox}_{\beta |\cdot|}(\|x\|) \frac{x}{\|x\|} \quad \mbox{and} \quad
\mathrm{prox}_{\beta f_\alpha}(x) = \mathrm{prox}_{\beta |\cdot|_\alpha}(\|x\|) \frac{x}{\|x\|}
$$
with the convention $\frac{0}{\|0\|}=0$.
\end{lemma}
\begin{proof} The derivation of $\mathrm{prox}_{\beta f}$ can be found in \cite{Bauschke-Combettes:11}. A direct computation (for example, see \cite{Micchelli-Shen-Xu:IP-11}) gives $\mathrm{env}_{\alpha} f (x) = \mathrm{env}_{\alpha} |\cdot| (\|x\|)$. Therefore, $f_\alpha(x)=\|x\|-\mathrm{env}_{\alpha} |\cdot| (\|x\|)=|\cdot|_\alpha(\|x\|)$. This function is isotropic, meaning it depends only on the magnitude of its argument and not the direction. Then for any $x \in \mathbb{R}^d$, every element of $\mathrm{prox}_{\beta f_\alpha}(x)$ should be a multiple of $\frac{x}{\|x\|}$. Moreover, we have
\begin{eqnarray*}
\mathrm{prox}_{\beta f_\alpha}(x)&=& \mathrm{arg}\min \left\{f_\alpha(w)+\frac{1}{2\beta} \|w-x\|^2: w\in\mathbb{R}^d\right\}\\
&=& \frac{x}{\|x\|} \cdot \mathrm{arg}\min \left\{|\cdot|_\alpha(\tau)+\frac{1}{2\beta} (\tau-\|x\|)^2: \tau\in\mathbb{R}\right\}\\
&=& \frac{x}{\|x\|} \cdot \mathrm{prox}_{\beta |\cdot|_\alpha}(\|x\|).
\end{eqnarray*}
This completes the proof.
\end{proof}

With this lemma, the expression of $\phee_\alpha$ is given in the following.
\begin{proposition}\label{prop:properties-ex2-falpha}
Let $\phee$ in \eqref{def:exampe2} be a compositional norm  on $\mathbb{R}^d$. For any $x \in \mathbb{R}^d$ and q positive parameter $\alpha$, we have that
\begin{equation}\label{falpha-ex2}
\phee_\alpha(x) = \sum_{j=1}^J |\cdot|_\alpha(\|I_{\omega_j} x\|).
\end{equation}
Furthermore, for any positive parameter $\beta$, we have that
\begin{equation}\label{prox-ex2}
\mathrm{prox}_{\beta \phee}(x) = \sum_{j=1}^J I^\top_{\omega_j}\mathrm{prox}_{\beta |\cdot|}(\|I_{\omega_j}x\|) \frac{I_{\omega_j}x}{\|I_{\omega_j}x\|}.
\end{equation}
and
\begin{equation}\label{prox-falpha-ex2}
\mathrm{prox}_{\beta \phee_\alpha}(x) = \sum_{j=1}^J I^\top_{\omega_j}\mathrm{prox}_{\beta |\cdot|_\alpha}(\|I_{\omega_j}x\|) \frac{I_{\omega_j}x}{\|I_{\omega_j}x\|}.
\end{equation}
\end{proposition}
\begin{proof}\ \ We omit the proof of equation~\eqref{prox-ex2} here since its proof is similar to that of equation~\eqref{prox-falpha-ex2}. Because of the block structure of $\phee$ given in \eqref{def:exampe2}, and using the definition of Moreau envelope, we have that
\begin{equation*}
\mathrm{env}_{\alpha}\phee(x)= \sum_{j=1}^J \min\left\{\frac{1}{2\alpha}\|v-I_{\omega_i}x\|^2+\|v\|: v \in \mathbb{R}^{\#\omega_i}\right\} = \sum_{j=1}^J \mathrm{env}_{\alpha}|\cdot|(\|I_{\omega_i}x\|).
\end{equation*}
From the above equation, we have $$\phee_\alpha(x)=\sum_{j=1}^J (\|I_{\omega_i}x\|-\mathrm{env}_{\alpha}|\cdot|(\|I_{\omega_i}x\|))=\sum_{j=1}^J |\cdot|_\alpha(\|I_{\omega_j} x\|),$$
which is \eqref{falpha-ex2}.

Next, we compute the proximity operator of $\phee_\alpha$. We have that
\begin{eqnarray*}
\mathrm{prox}_{\beta \phee_\alpha}(x)&=& \mathrm{arg}\min \left\{\sum_{j=1}^J \left(|\cdot|_\alpha(\|I_{\omega_j} w\|)+\frac{1}{2\beta} \|I_{\omega_j}w-I_{\omega_j}x\|^2\right): w\in\mathbb{R}^d\right\}\\
&=&\sum_{j=1}^J I^\top_{\omega_j} \mathrm{arg}\min \left\{\left(|\cdot|_\alpha(\|u\|)+\frac{1}{2\beta} \|u-I_{\omega_j}x\|^2\right): u\in\mathbb{R}^{\#\omega_j}\right\},
\end{eqnarray*}
which is \eqref{prox-falpha-ex2} by Lemma~\ref{lemma:L1-L2}.
\end{proof}


\section{Problem Formulation and Algorithms}\label{sec:Algorithms}

We are now able to state the model under consideration in this paper and to provide some insight into its benefits. We are interested in solving
\begin{equation}\label{eq:Primal}\tag{$\mathcal{P}$}
\argmin \left\{ ~ \frac{1}{2\lambda}\|x-z\|^2 + \phee_\alpha(Bx) ~: x \in C \right\},
\end{equation}
where $C \subset \mathbb{R}^d$ is closed and convex, $z \in \mathbb{R}^d$, $B \in \mathbb{R}^{n \times d}$, and $\phee_\alpha$ is a SPF as defined by \eqref{defn:falpha}.

While in some instances $\phee_\alpha$ may be convex (e.g. if $\phee$ is sufficiently strongly convex), without further information, we regard $\phee_\alpha$ as nonconvex. However, depending on the parameters $\alpha$ and $\lambda$ as well as the choice of matrix $B$, we see that  \eqref{eq:Primal} may be convex.
\begin{lemma}\label{lem:Convexity}
For any convex function $\varphi$ on $\mathbb{R}^n$, an $n \times d$ matrix $B$, positive parameters $\lambda$ and $\alpha$, and any fixed
$z \in \mathbb{R}^n$, define
$$
W(x):=\frac{1}{2\lambda}\|x-z\|^2+\varphi_\alpha(Bx).
$$
If $\lambda < \frac{\alpha}{\|B\|^2}$, then $W$ is strictly convex on $\mathbb{R}^n$. If $\lambda = \frac{\alpha}{\|B\|^2}$, then $W$ is convex.
\end{lemma}
\begin{proof}
This is a direct consequence of item (iv) in Lemma~\ref{lemma:properties}.
\end{proof}

Clearly, the above lemma tells us that for $\lambda \leq \frac{\alpha}{\|B\|^2}$, any critical point of \eqref{eq:Primal} is a global minimum, and for $\lambda < \frac{\alpha}{\|B\|^2}$, the minimizer is unique.

The structure of $\phee_\alpha$ lends flexibility to this model; \eqref{eq:Primal} can be made to fit a variety of generic models, both convex and nonconvex, by grouping terms in different ways.  For example, the objective function of the model can be viewed as a difference of convex functions $\frac{1}{2\lambda} \|x-z\|^2 + \phee(Bx)$  and $\env_\alpha \phee(Bx)$,  and therefore suitable for the rich framework of DC algorithms \cite{An-Tao:AOR:2005}.

Based on the structure of $\phee_\alpha$, we can decompose  model \eqref{eq:Primal} in three ways which correspond to different classes of algorithms: convex, difference of convex, and nonconvex. More precisely, we study three algorithms which highlight each of these cases: primal-dual splitting (PD), difference of convex (DC), and primal-dual hybrid  gradient (PDHG) method.



\subsection{Primal Dual Splitting}
By identifying
\begin{equation}\label{eq:FGH}
    F(x) = \frac{1}{2\lambda} \|x-z\|_2^2 - \env_\alpha \phee (Bx), \quad G(x) = \iota_C(x), \quad H(x) = \phee(x),
\end{equation}
model \eqref{eq:Primal} can be viewed as a special case of the following generic model
\begin{equation}\label{eq:PrimalDual_Generic_Model}
\mathrm{arg}\min \left\{F(x) + G(x) + H(Bx): x \in \mathbb{R}^d\right\}.
\end{equation}
Under the assumptions that (i) $F$ is convex and differentiable with $L$-Lipschitz gradient and (ii) $G$ and $H$ are proper, convex, lower semicontinuous, and prox-friendly, a primal-dual splitting algorithm, proposed in \cite{Condat:JOTA:2013} for \eqref{eq:PrimalDual_Generic_Model}, is as follows: given initial points $(\x{0}, \y{0})$ and positive parameters $\sigma, \tau, \rho$, iterate
\begin{align}
\xx{k+1} &\coloneqq \prox_{\tau G}(\x{k} - \tau \nabla F(\x{k}) - \tau B^\top \y{k})\label{alg:PD1} \\
\yy{k+1} & \coloneqq \prox_{\sigma H^*}(\y{k} + \sigma B(2\xx{k+1} - \x{k})) \label{alg:PD2}\\
\begin{bmatrix}\x{k+1}\\ \y{k+1}\end{bmatrix} & \coloneqq \rho\begin{bmatrix} \xx{k+1} \\ \yy{k+1} \end{bmatrix} + (1-\rho)\begin{bmatrix}\x{k} \\ \y{k}\end{bmatrix}.\label{alg:PD3}
\end{align}
In the above scheme, $H^*$ is the Fenchel conjugate of $H$. The convergence analysis of the above iterative scheme given in \cite{Condat:JOTA:2013} is stated in the following result.
\begin{proposition}[Condat \cite{Condat:JOTA:2013}]\label{prop:condat}
Let $\tau$, $\sigma$, and $\rho$ be the parameters in \eqref{alg:PD1}--\eqref{alg:PD3}. Suppose that the functions $F$, $G$, and $H$ in \eqref{eq:PrimalDual_Generic_Model} are convex, the gradient of $F$ is $L$-Lipschitz with $L>0$, and the following hold: (i) $\frac{1}{\tau}-\sigma \|B\|^2 > \frac{L}{2}$; and (ii) $\rho \in (0, 1]$. Then the sequence $(\x{k})_{k\in \mathbb{N}}$ converges to a solution of the problem~\eqref{eq:PrimalDual_Generic_Model}.
\end{proposition}

We now verify the assumptions of Proposition \ref{prop:condat} through the identifications~\eqref{eq:FGH}.

\begin{proposition}\label{prop:Lips}
Let $F$ be defined as in \eqref{eq:FGH}. Then the following statements hold:
\begin{enumerate}
\item $F$ is differentiable. Moreover, its gradient is $L$-Lipschitz continuous with
$$ L=\left\{
       \begin{array}{ll}
         \frac{1}{\lambda}, & \hbox{if $\|B\|^2 \le \frac{2\alpha}{\lambda}$;} \\
         \sqrt{\frac{1}{\lambda^2}+\frac{\|B\|^2}{\alpha^2} (\|B\|^2-\frac{2\alpha}{\lambda})}, & \hbox{otherwise.}
       \end{array}
     \right.
$$
\item $F$ is strictly convex on $\mathbb{R}^n$ if $\lambda < \frac{\alpha}{\|B\|^2}$; convex if $\lambda = \frac{\alpha}{\|B\|^2}$.
\end{enumerate}
\end{proposition}
\begin{proof}
(i): We know that Moreau envelope of a convex function is differentiable. Hence, $F$ is differentiable and is simply the difference of two differentiable functions. Actually, we have that
$$
\nabla F= \frac{1}{\lambda}(\cdot-z)-B^\top \mathrm{prox}_{\alpha^{-1} \varphi^*} (\alpha^{-1} B\cdot).
$$
For any $x$ and $y$ in $\mathbb{R}^n$, let us denote  $p=\mathrm{prox}_{\alpha^{-1} \varphi^*} (\alpha^{-1} Bx)$ and $q=\mathrm{prox}_{\alpha^{-1} \varphi^*} (\alpha^{-1} By)$. Then, one has
\begin{eqnarray*}
\|\nabla F(x)-\nabla F(y)\|^2&=&
\frac{1}{\lambda^2}\|x-y\|^2-\frac{2\alpha}{\lambda}\langle \alpha^{-1}B(x-y),  p-q\rangle + \|B^\top(p-q)\|^2\\
&\le& \frac{1}{\lambda^2}\|x-y\|^2-\frac{2\alpha}{\lambda}\|p-q\|^2 + \|B^\top(p-q)\|^2\\
&=&\frac{1}{\lambda^2}\|x-y\|^2 + (p-q)^\top (BB^\top-\frac{2\alpha}{\lambda}\id)(p-q).
\end{eqnarray*}
Obviously, if $\|B\|^2 \le \frac{2\alpha}{\lambda}$, then $BB^\top-\frac{2\alpha}{\lambda}\id$ is semi-negative. Thus, $\|\nabla F(x)-\nabla F(y)\| \le \frac{1}{\lambda^2}\|x-y\|$. If $\|B\|^2 > \frac{2\alpha}{\lambda}$, then, by using the inequality $\|p-q\|\le \alpha^{-1}\|B\| \|x-y\|$, we have
$$
(p-q)^\top (BB^\top-\frac{2\alpha}{\lambda}\id)(p-q)\le  \frac{\|B\|^2}{\alpha^2}({\|B\|^2}-\frac{2\alpha}{\lambda})\|x-y\|^2.
$$
The result follows immediately.

(ii): By the definition of the Moreau envelope, we have
    \begin{align*}
        F(x) &= \frac{1}{2\lambda}\|x-z\|^2 - \min \left\{ \frac{1}{2\alpha}\|u - Bx\|^2 + \phee(u) : u \in \mathbb{R}^n \right\}\\
        & = \frac{1}{2\lambda} \|x-z\|^2 - \frac{1}{2\alpha}\|Bx\|^2 + \frac{1}{2\alpha} \max \left\{ 2 \langle B^\top u, x \rangle - \|u\|^2 - 2\alpha \phee(u) : u\in \mathbb{R}^n \right\}.
    \end{align*}
    Since
    $$ \frac{1}{2\lambda}\|x-z\|^2 - \frac{1}{2\alpha}\|Bx\|^2 = x^\top \left(\frac{1}{2\lambda} \id - \frac{1}{2\alpha} B^\top B\right) x + \frac{1}{2\lambda}(\|z\|^2 - 2z^\top x),$$
    which is strictly convex if $\lambda < \frac{\alpha}{\|B\|^2}$, and $\max \left\{ 2 \langle B^\top u, x \rangle - \|u\|^2 - 2\alpha \phee(u) : u\in \mathbb{R}^n \right\}$ is convex as a function of $x$, we see that $F$ is strictly convex. Finally, if $\lambda = \frac{\alpha}{\|B\|^2}$, it is clear that $F$ is convex.
\end{proof}

Algorithm \ref{alg:PD} is the direct application of \eqref{alg:PD1}-\eqref{alg:PD3} to \eqref{eq:Primal} through the identifications \eqref{eq:FGH}. We note that for the given $F$,
\begin{equation*}
    \nabla F(x) = \frac{1}{\lambda}(x-z) - B^\top \nabla \env_\alpha \phee(Bx).
\end{equation*}
Applying the Moreau Identity, we write $\nabla \env_\alpha \phee (Bx) = \prox_{\alpha^{-1} \phee^*} (\alpha^{-1}Bx).$

\begin{algorithm}[h]
\caption{Primal-Dual Splitting Algorithm for \eqref{eq:Primal}}\label{alg:PD}
\KwIn{
   {Initialization: Choose the positive parameters $\tau$, $\sigma$, the sequence of positive relaxation parameters $(\rho_n)_{n \in \mathbb{N}}$ and the initial estimates $\x{0} \in \mathbb{R}^d$, $\y{0} \in \mathbb{R}^n$.}
   }
\For{$n=0, 1, \ldots$}
{
\begin{eqnarray*}
\xx{k+1} &\leftarrow& \mathrm{proj}_C\left(\x{k}-\tau \left(\frac{1}{\lambda}(\x{k}-z)\right) + \tau B^\top \left(\prox_{\alpha^{-1} \varphi^*}(\alpha^{-1} B\x{k}) - \y{k}\right) \right)\\
\yy{k+1} &\leftarrow& \prox_{\sigma \varphi^*}\left(\y{k}+\sigma B(2\xx{k+1}-\x{k})\right)\\
\begin{bmatrix}\x{k+1}\\ \y{k+1}\end{bmatrix} &\leftarrow& \rho\begin{bmatrix} \xx{k+1} \\ \yy{k+1} \end{bmatrix} + (1-\rho)\begin{bmatrix}\x{k} \\ \y{k}\end{bmatrix}
\end{eqnarray*}
}
\end{algorithm}

\begin{theorem}\label{cor:PD_convergence}
Let $\lambda$, $\alpha$, and $z$ be as in problem~\eqref{eq:Primal}, and let $\tau$, $\sigma$, and $\rho$ be the parameters in Algorithm~\ref{alg:PD}. Suppose that $\lambda < \frac{\alpha}{\|B\|^2}$ and the following hold:
\begin{itemize}
  \item[(i)] $\frac{1}{\tau}-\sigma \|B\|^2 > \frac{1}{2\lambda}$;
  \item[(ii)] $\rho \in (0,1]$.
\end{itemize}
Then the sequence $(\x{k})_{k\in \mathbb{N}}$  produced by Algorithm \ref{alg:PD} converges to a solution of the problem~\eqref{eq:Primal}.
\end{theorem}
\begin{proof}\ \ By Lemma~\ref{lem:Convexity} and Proposition~\ref{prop:Lips}, if $\lambda < \frac{\alpha}{\|B\|^2}$, then the objective function of problem~\eqref{eq:Primal} is strictly convex, and the gradient of $F$ given in \eqref{eq:FGH} is $\frac{1}{\lambda}$-Lipschitz continuous. Hence, the convergence of the sequence  $(\x{k})_{k\in \mathbb{N}}$ is the consequence of Proposition~\ref{prop:condat}.
\end{proof}


\subsection{Difference of Convex Algorithm}
By $\varphi_\alpha=\varphi-\mathrm{env}_\alpha\varphi$ from \eqref{defn:falpha}, set
\begin{equation}\label{eq:FGH-DC}
Q(x)=\frac{1}{2\lambda} \|x-z\|^2_2+\iota_C(x)+\varphi(Bx), \quad
P(x)=\mathrm{env}_\alpha(Bx),
\end{equation}
then model~\eqref{eq:Primal} can be viewed as a special case of the following generic model
\begin{equation}\label{eq:ModelDC}
\min \{Q(x) - P(x): x \in \mathbb{R}^d\},
\end{equation}
where both $P$ and $Q$ are convex functions. Due the objective function is the difference of convex (DC) functions, model~\eqref{eq:ModelDC} is referred  to as DC program.

DCA (DC algorithm) is based on local optimality conditions and duality in DC programming \cite{LeThi:MP:2018}. The main idea of DCA is as follow: at each iteration $k$, DCA approximates the second DC component $P(x)$ by the affine approximation $P_k(x)=P(\x{k})+\langle \y{k}, x-\x{k}\rangle$, with $\y{k} \in \partial P(\x{k})$, and minimizes the resulting convex function. DCA for \eqref{eq:ModelDC} is as follows:
\begin{eqnarray}
\y{k}  & \in & \partial P(\x{k}) \label{eq:BacisDC-1}\\
\x{k+1}  & \in & \arg\min \{ Q(x)- P_k(x): x\in \mathbb{R}^d\} \label{eq:BacisDC-2}
\end{eqnarray}
As the optimal solution set of \eqref{eq:BacisDC-2} is $\partial Q^*(\y{k})$, the DCA scheme can be expressed in another form:
$$
\mbox{For $k=0,1,\ldots$, set}\quad \y{k}   \in  \partial P(\x{k}); \quad \x{k+1} \in \partial Q^*(\y{k}).
$$
We state the local convergence properties of DCA in the following theorem (see \cite{Tao-An:SIAMOPT:1998}).
\begin{theorem}[\cite{Tao-An:SIAMOPT:1998},Theorem 3.7]\label{thm:DC-Orig}
Suppose that the sequence $\{\x{k}\}_{k\in \mathbb{N}}$ is defined by the iterative scheme \eqref{eq:BacisDC-1}-\eqref{eq:BacisDC-2} for problem~\eqref{eq:ModelDC}. Then we have
\begin{itemize}
\item[(i)] The objective value sequence $\{Q(\x{k})-P(\x{k})\}_{k\in \mathbb{N}}$ is monotonically decreasing.
\item[(ii)] If the optimal value of problem~\eqref{eq:ModelDC} is finite and the sequence $\{\x{k}\}_{k\in \mathbb{N}}$ is bounded, then every limit point $x^\diamond$ of $\{\x{k}\}_{k\in \mathbb{N}}$ is a critical point of the problem.
\end{itemize}
\end{theorem}

With these properties on DC programming in hands, we turn back to the problem~\eqref{eq:ModelDC} with $P$ and $Q$ given in \eqref{eq:FGH-DC}.
\begin{algorithm}\label{alg:DCA}
\caption{DCA  scheme for \eqref{eq:ModelDC} with $P$ and $Q$ given in \eqref{eq:FGH-DC}}
\KwIn{
	{Choose $\x{0} \in \mathrm{dom} \partial P$, $k=0$}
	}
\For{$k = 0, 1, \dots$}
{
\begin{eqnarray}
\y{k}  & \leftarrow & B^\top \nabla \env_\alpha \phee(B \x{k}) \label{eq:DC-1}\\
\x{k+1}  & \leftarrow & \arg\min \left\{\frac{1}{2\lambda}\|x-z\|^2+\iota_C(x)+\varphi(Bx)- \langle \y{k},x \rangle: x\in \mathbb{R}^d\right\} \label{eq:DC-2}
\end{eqnarray}
}
\end{algorithm}

\begin{theorem}\label{thm:DC}
Suppose that the sequences $\{\x{k}\}_{k\in \mathbb{N}}$ and $\{\y{k}\}_{k\in \mathbb{N}}$ are generated by Algorithm~\ref{alg:DCA} for problem~\eqref{eq:ModelDC} with $P$ and $Q$ given in \eqref{eq:FGH-DC}. Then every limit point $x^\diamond$ of $\{\x{k}\}_{k\in \mathbb{N}}$ is a critical point of the problem. Moreover, $\lim_{k \rightarrow \infty}\|\x{k+1}-\x{k}\|=0$.
\end{theorem}
\begin{proof}\ \ Recall that $Q(x)-P(x)=\frac{1}{2\lambda} \|x-z\|^2 +\iota_C(x)+\varphi_\alpha(Bx)$ which is nonnegative and continuous on its domain. Hence, the optimal value of problem~\eqref{eq:ModelDC} is finite. From item (i) of Theorem~\ref{thm:DC-Orig}, we have that
$$
\frac{1}{2\lambda} \|\x{k}-z\|^2 \le Q(\x{k})-P(\x{k}) \le Q(\x{0})-P(\x{0}) <\infty,
$$
it leads to the boundedness of the sequence $\{\x{k}\}_{k\in \mathbb{N}}$. From~\eqref{eq:DC-1} and the fact that $\id-\mathrm{prox}_{\alpha \varphi}$ is nonexpasive operator, we have $\|\y{k}\|= \frac{1}{\alpha}\|B^\top(\id-\mathrm{prox}_{\alpha \varphi})(B\x{k})\| \le \frac{\|B\|^2}{\alpha}\|\x{k}\|$,
hence the  $\{\y{k}\}_{k\in \mathbb{N}}$ is bounded. By item (ii) of Theorem~\ref{thm:DC-Orig}, we know that every limit point $x^\diamond$ of $\{\x{k}\}_{k\in \mathbb{N}}$ is a critical point of the problem.

By $\y{k} \in \partial P(\x{k})$, we have $P(\x{k+1}) \geq P(\x{k}) + \langle \y{k}, \x{k+1} - \x{k} \rangle.$ Since $Q$ is strongly convex and $\x{k+1}$ minimizes $Q(x) - \langle \y{k}, x \rangle$, we get
\[ Q(\x{k+1}) - \langle \y{k}, \x{k+1} \rangle \leq Q(\x{k}) - \langle \y{k}, \x{k} \rangle - \frac{1}{2\lambda}\|\x{k+1} - \x{k}\|^2. \]
Therefore, it follows that
\begin{eqnarray*}
&&Q(\x{k+1})-P(\x{k+1})\\
&\le& Q(\x{k+1}) -\left(P(\x{k})+\langle \y{k}, \x{k+1}-\x{k}\rangle\right)\\
&=& Q(\x{k+1}-\langle \y{k}, \x{k+1}\rangle - \left(P(\x{k})-\langle \y{k}, \x{k}\rangle\right) \\
&\le& Q(\x{k}) -\langle \y{k}, \x{k}\rangle-\frac{1}{2\lambda}\|\x{k+1}-\x{k}\|^2 - \left(P(\x{k})-\langle \y{k}, \x{k}\rangle\right)\\
&=&Q(\x{k})-P(\x{k})-\frac{1}{2\lambda}\|\x{k+1}-\x{k}\|^2.
\end{eqnarray*}
From this, we get
$$
\frac{1}{2\lambda}\|\x{k+1}-\x{k}\|^2 \le (Q(\x{k})-P(\x{k}))- (Q(\x{k+1})-P(\x{k+1})).
$$
Summing the above inequality for all $k$ from $0$ to infinity yields
$$
\frac{1}{2\lambda}\sum_{k=0}^\infty\|\x{k+1}-\x{k}\|^2 \le Q(\x{0})-P(\x{0}),
$$
which implies $\lim_{k \rightarrow \infty}\|\x{k+1}-\x{k}\|=0$.
\end{proof}



\subsection{Primal-Dual Hybrid Gradient Methods}
Set
\begin{equation}\label{eq:PQ-HPDG}
Q(x)=\frac{1}{2\lambda} \|x-z\|^2_2+\iota_C(x), \quad
P(x)=\varphi_\alpha(x),
\end{equation}
then model~\eqref{eq:Primal} can be viewed as a special case of the following generic model
\begin{equation}\label{eq:Model-HPD}
\min \{Q(x) + P(Bx): x \in \mathbb{R}^d\},
\end{equation}
where $P$ is semiconvex and $Q$ is convex. In this setting, a  primal-dual hybrid gradient (PDHG) method was proposed for  model~\eqref{eq:Model-HPD} in \cite{Mollenhoff:Strekalovskiy:Moeller:Cremers:SIAMIS:2015} as follows: Given a pair $(\x{0},\ttheta{0}) \in \mathbb{R}^d \times \mathbb{R}^n$ and for $\bar{x}^{(0)}=\x{0}$, $\sigma>0$, $\tau>0$, and $\rho \in [0,1]$, iterate for all $k\ge 0$
\begin{eqnarray}
  \uu{k+1} &=& \mathrm{argmin}\left\{ \frac{\sigma}{2}\|u-B \bar{x}^{(k)}\|^2-\langle u, \ttheta{k} \rangle + P(u): u \in \mathbb{R}^d \right\} \label{tmp1:Alg:HPD} \\
  \ttheta{k+1} &=& \ttheta{k}+\sigma(B\bar{x}^{(k)}-\uu{k+1})  \label{tmp2:Alg:HPD}  \\
  \x{k+1} &=& \mathrm{argmin} \left\{\frac{1}{2\tau} \|x-\x{k}\|^2+\langle Bx, \ttheta{k+1} \rangle +Q(x): x \in  \mathbb{R}^n \right\} \label{tmp3:Alg:HPD}  \\
  \bar{x}^{(k+1)} &=& \x{k+1}+\rho (\x{k+1}-\x{k}) \label{tmp4:Alg:HPD}
\end{eqnarray}

We first show that the solution to the minimization problem~\eqref{tmp1:Alg:HPD} can be explicitly given as follows:
\begin{equation}\label{tmp1:Alg:HPD-sln}
\uu{k+1} = \prox_{\sigma^{-1} \varphi_\alpha}\left(B\bar{x}^{(k)}+\frac{1}{\sigma}\ttheta{k} \right)
\end{equation}
Next, the solution to the minimization problem~\eqref{tmp3:Alg:HPD} is the solution of the following linear system
3`\begin{equation}\label{tmp3:Alg:HPD-sln}
\x{k+1} = \mathrm{proj}_C\left(\frac{\lambda}{\tau+\lambda} \x{k} + \frac{\tau}{\tau+\lambda}z - \frac{\tau \lambda}{\tau+\lambda} B^\top \ttheta{k+1}  \right)
\end{equation}

Below we give our PDHG-based algorithm for solving the optimization problem~\eqref{eq:Model-HPD}.
\begin{algorithm}\label{alg:HPDG}
\caption{PDHG scheme for problem~\eqref{eq:Model-HPD}}
\KwIn{
	{$z$, $\lambda>0$, $\alpha>0$, $\rho \in [0,1]$. Initialize $\x{0}=z$,  $\ttheta{0}=0$, $\bar{x}^{(0)}=\x{0}$}
	}
\For{$k = 0, 1, \dots$}
{
\begin{itemize}
\item[1)]  $\uu{k+1} \leftarrow \prox_{\sigma^{-1} \varphi_\alpha}\left(B\bar{x}^{(k)}+\frac{1}{\sigma}\ttheta{k} \right)$
\item[2)] $\ttheta{k+1} \leftarrow \ttheta{k}+\sigma(B\bar{x}^{(k)}-\uu{k+1})$
\item[2)] $\x{k+1} \leftarrow \mathrm{proj}_C\left(\frac{\lambda}{\tau+\lambda} \x{k} + \frac{\tau}{\tau+\lambda}z - \frac{\tau \lambda}{\tau+\lambda} B^\top \ttheta{k+1}  \right)$
\item[3)] $\bar{x}^{(k+1)} \leftarrow \x{k+1}+\rho (\x{k+1}-\x{k})$
\end{itemize}
}
\end{algorithm}

\begin{theorem}\label{thm:HPDG}
For optimization model~\eqref{eq:Model-HPD} with $P$ and $Q$ given in \eqref{eq:PQ-HPDG}, if $\alpha \ge \lambda \|B\|^2$, then Algorithm~\ref{alg:HPDG} converges the unique solution $x^\star$ of model~\eqref{eq:Model-HPD} for $\sigma \alpha=2$, $\tau \sigma \|B\|^2 \le 1$, and any $\rho \in [0, 1]$, with rate $\|\x{k}-x^\star\|^2 \le \widetilde{C}/n$ for some constant $\widetilde{C}$.
\end{theorem}
\begin{proof}\ \ As we know, $P$ is $\frac{1}{\alpha}$-semiconvex and $Q$ is $\frac{1}{\lambda}$-strongly convex. By Theorem~2.8 in \cite{Mollenhoff:Strekalovskiy:Moeller:Cremers:SIAMIS:2015}, the conclusion of this theorem holds for the given parameters $\sigma$, $\tau$.
\end{proof}


\subsection{Discussion}

As noted above, one of the main motivations for using nonconvex penalties is to avoid biased solutions. We now provide some discussion to show how this is accomplished in practice in each of the above algorithms. To illustrate these ideas, we look at the example of piecewise constant signals in $\mathbb{R}^d$. To be precise, we set $\phee = \|\cdot\|_1$, $C = \mathbb{R}^d$, and let $B$ be the one dimensional difference matrix. The vector $z \in \mathbb{R}^d$ is the noisy observation from which we hope to recover the true signal. Piecewise constant signals are sparse under the transformation $B$; in other words, all of the information about these signals is contained in the amplitude changes. When noise is added, the signal becomes nonsparse, though we assume that the noise is small compared to the signal. An example of such a signal and the noisy observation are given in Figure \ref{fig:signals}.

\begin{figure}[h]
\begin{center}
\begin{tabular}{cc}
\includegraphics[scale=0.25]{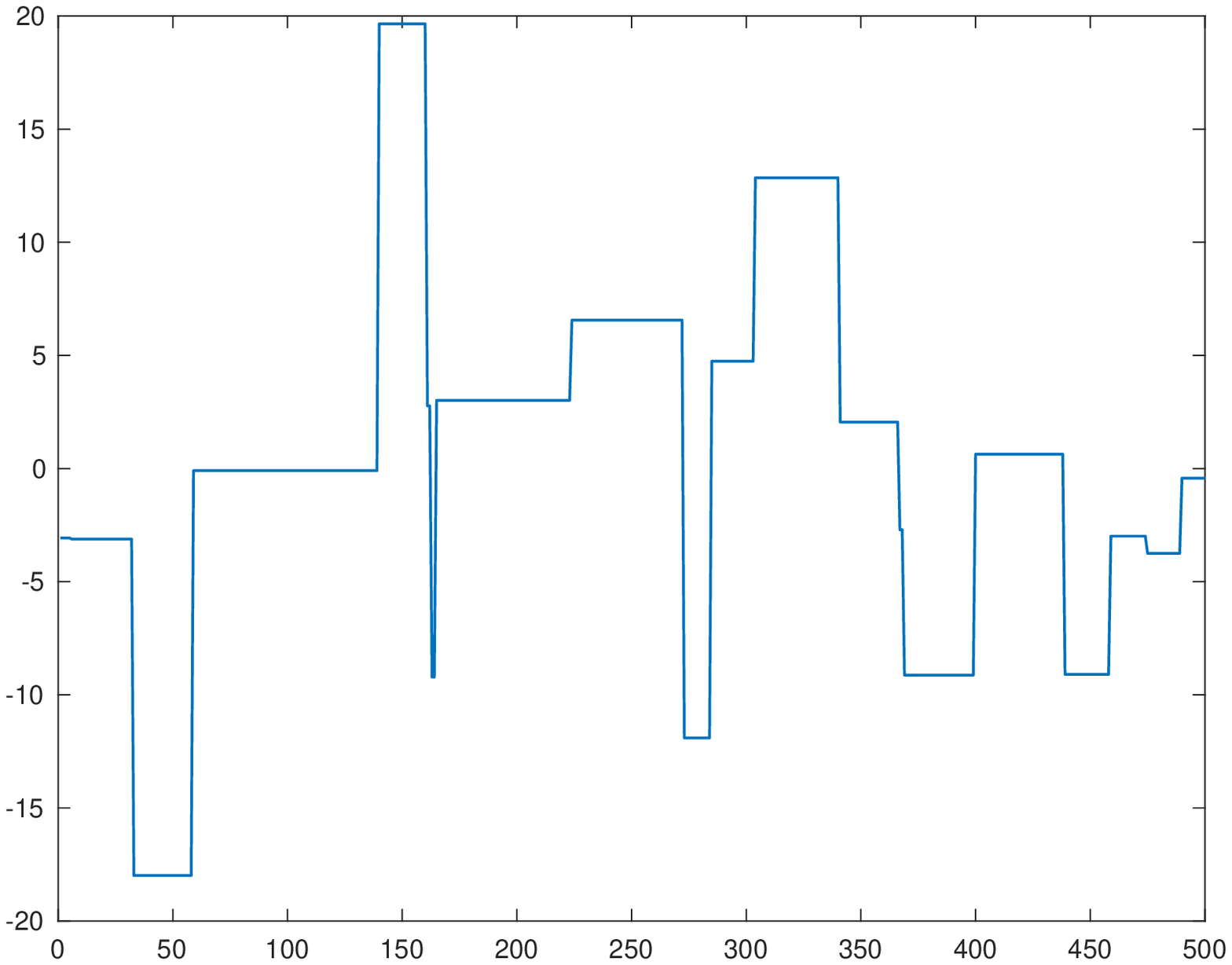} & \includegraphics[scale=0.25]{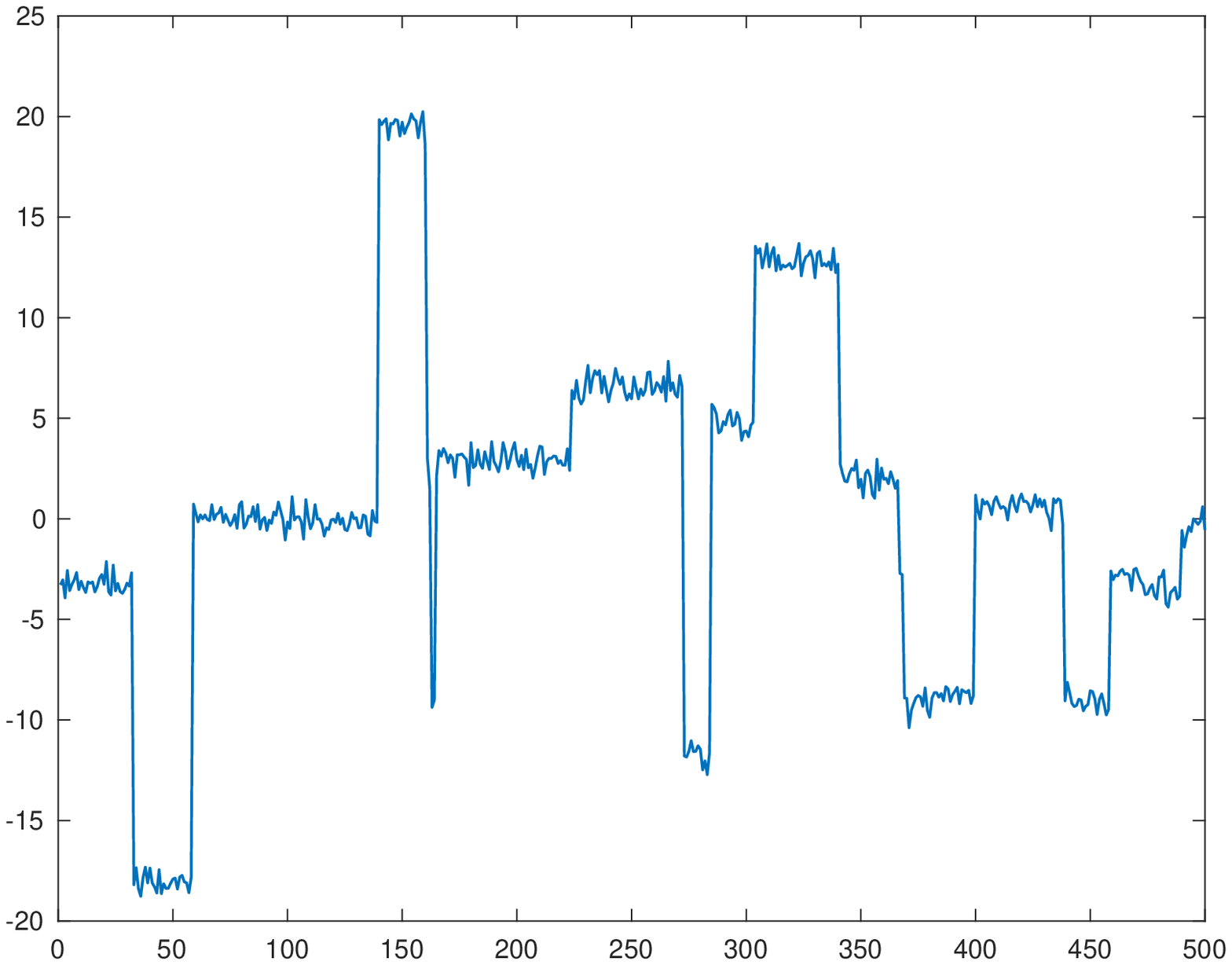} \\
(a) & (b) \\
\includegraphics[scale=0.25]{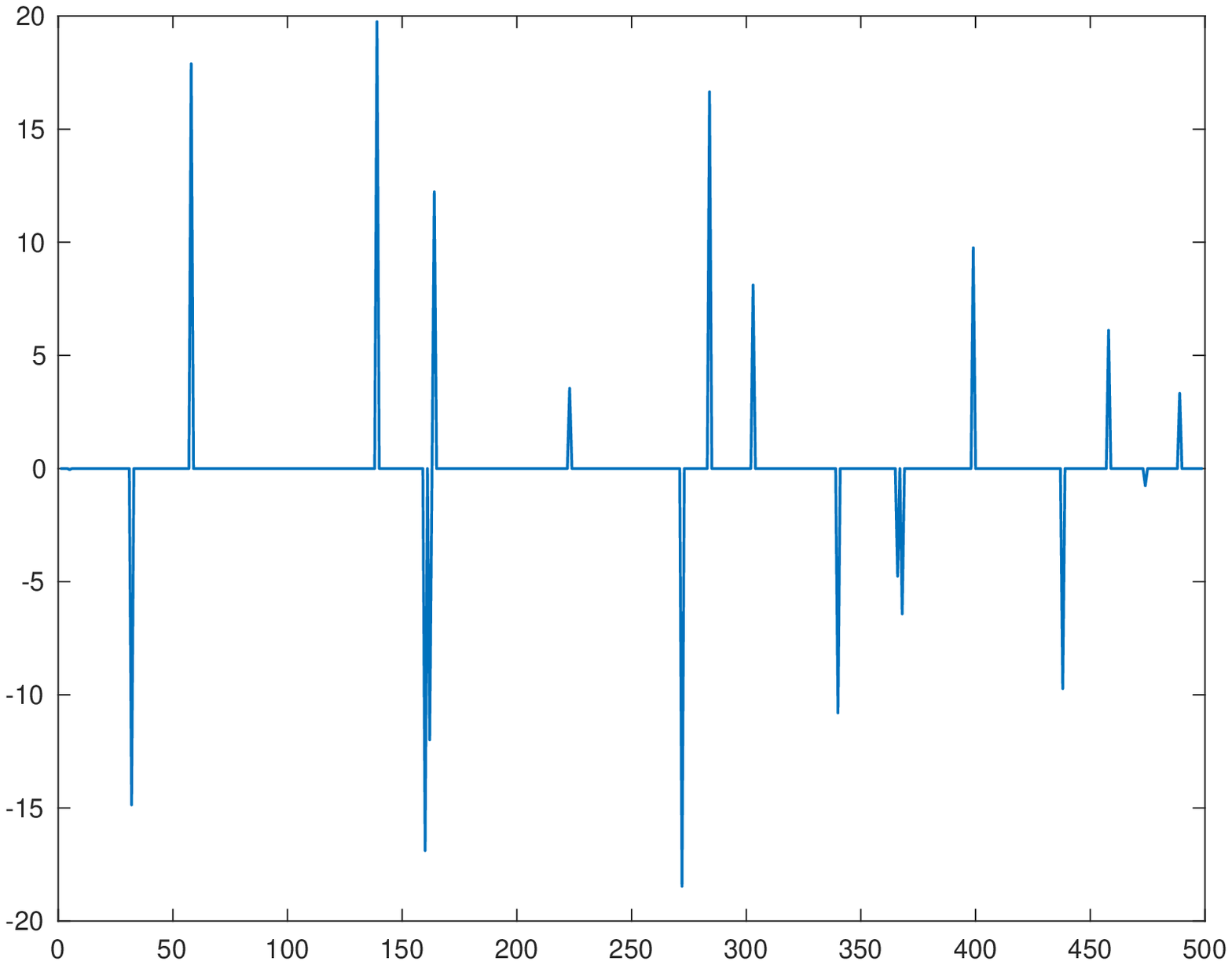} & \includegraphics[scale=0.25]{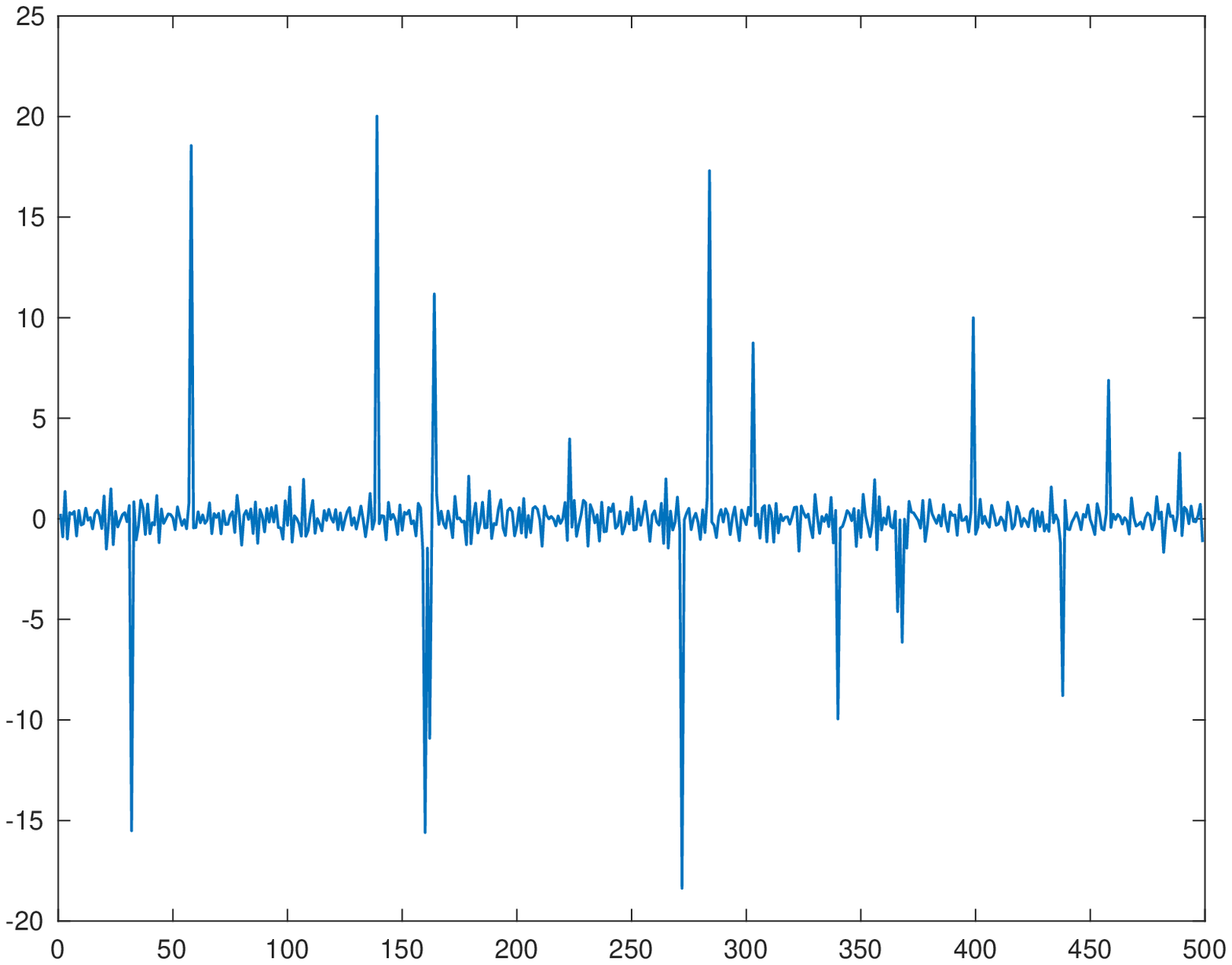}\\
(c) & (d)
\end{tabular}
\caption{(a) A piecewise constant signal $x$, (b) the signal with additive Gaussian noise $z$, (c) the sparse representation $Bx$, and (d) the nonsparse $Bz$.}
\label{fig:signals}
\end{center}
\end{figure}

In Algorithm \ref{alg:PD}, the primal updates are
\[ \xx{k+1} = \mathrm{proj}_C\left(\x{k}-\frac{\tau}{\lambda}(\x{k}-z)+ \tau B^\top \left(\nabla \env_\alpha \phee(B\x{k}) - \y{k}\right) \right) \]
where $\nabla (\env_\alpha \phee \circ B) (\x{k}) = B^\top \prox_{\alpha^{-1} \phee^*} (\alpha^{-1} B\x{k})$, as written in Section \ref{sec:Algorithms}. When $\phee = \|\cdot\|_1$, this term is projection of the differences of the current iterate onto the $\ell_\infty$ unit ball $\{ x \in \mathbb{R}^d : \|x\|_\infty \leq 1 \}$. This moves $\x{k}$ away from the set $\argmin \env_\alpha \phee \circ B = \argmin \phee \circ B$, which keeps relevant data from being pulled to zero.  As shown in Figure \ref{fig:PDupdates}, the addition of this term boosts the features of the current iterate in proportion to their magnitude, balancing the shrinkage enforced by the dual update.

\begin{figure}[H]
\begin{center}
\begin{tabular}{cc}
\includegraphics[scale=0.25]{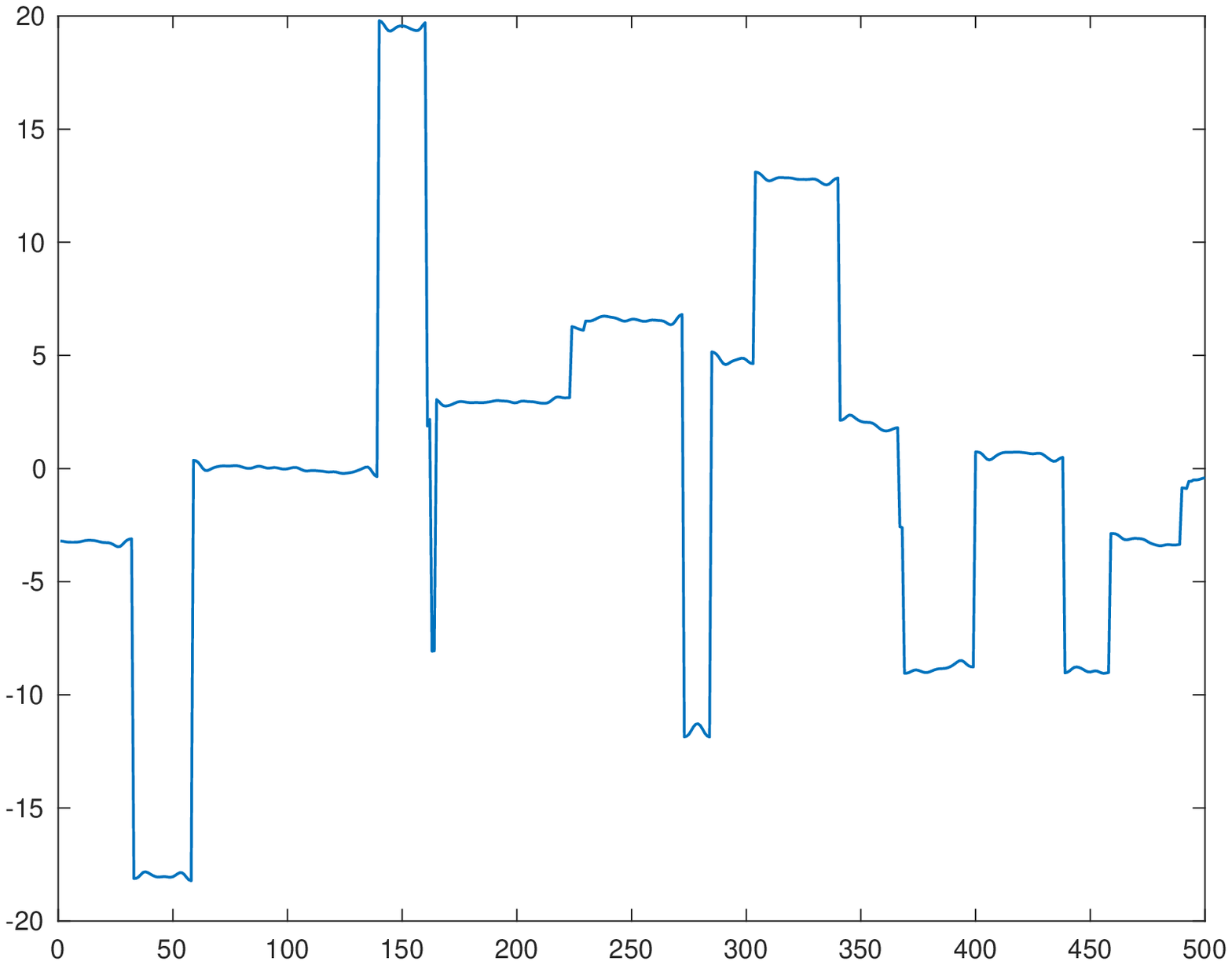} & \includegraphics[scale=0.25]{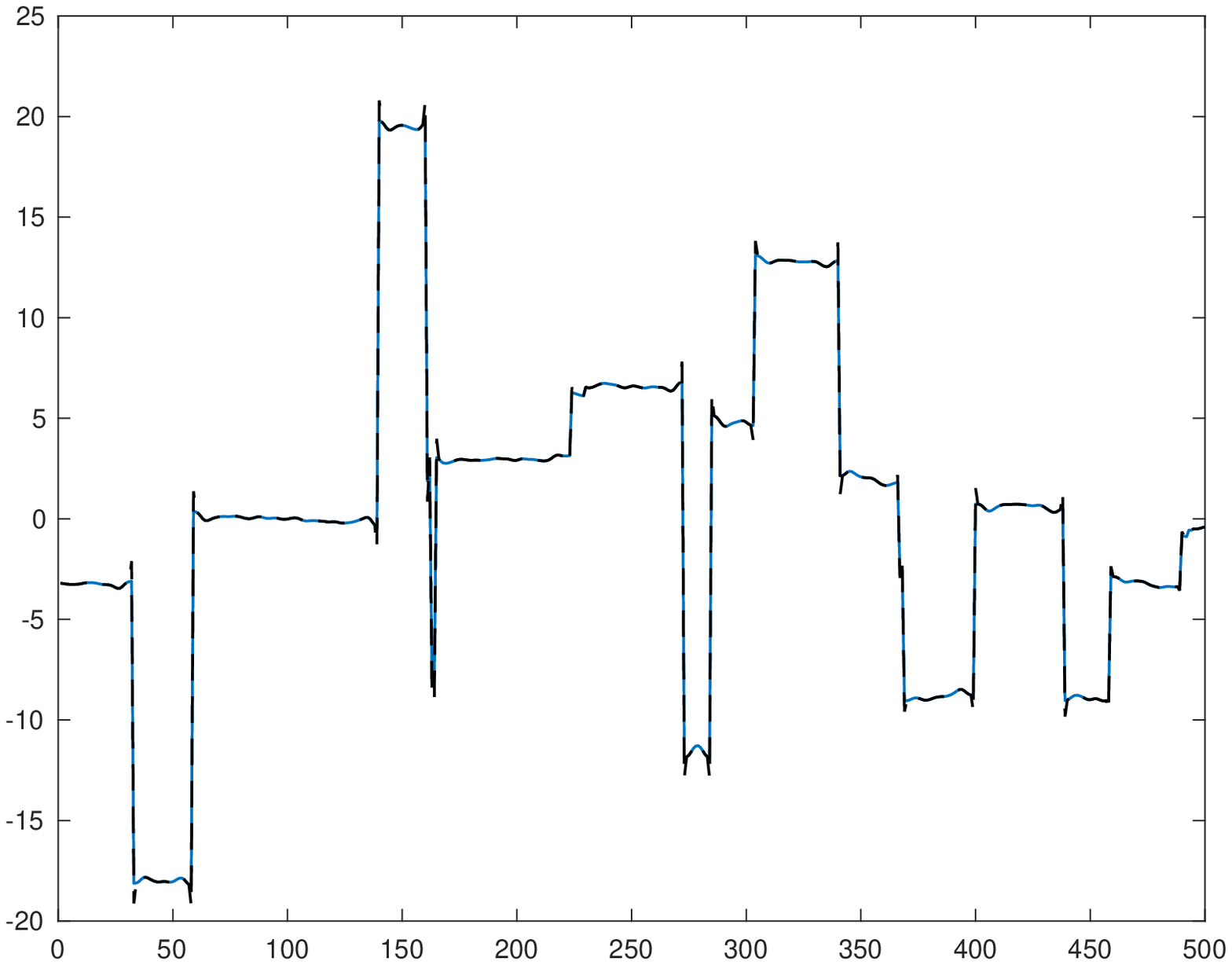}\\
(a) & (b) \\

\end{tabular}
\caption{Algorithm \ref{alg:PD}. (a) One iterate $\x{k}$ and (b) $\x{k} + B^\top \nabla \env_\alpha \phee(B \x{k})$ (dashed black) over $\x{k}$ (solid blue). The scaling factor $\tau$ is omitted for visibility.} \label{fig:PDupdates}
\end{center}
\end{figure}

Algorithm \ref{alg:DCA} requires solving a convex optimization problem in each iteration. Note that
\begin{align*} &\argmin \{ \phee(Bx) + \frac{1}{2\lambda}\|x-z\|^2 - \langle \nabla (\env_\alpha \phee \circ B)(\x{k}), x \rangle : x \in \mathbb{R}^d \} \\ = & \argmin \{ \phee(Bx) + \frac{1}{2\lambda}\|x - (z + \lambda \nabla (\env_\alpha \phee \circ B)(\x{k})\|^2 : x \in \mathbb{R}^d \}. \end{align*}
That is, this algorithm modifies the noisy signal at each iteration using the most recent update. This is very similar to Bregman iterations for solving the TV denoising problem with the subgradient of $\| B \cdot \|_1$ replaced by the gradient of the envelope (see \cite{Yin-Osher:Bregman:2008}). As before, this boosts the relevant features of the signal, as illustrated in Figure \ref{fig:DCupdates}.

\begin{figure}[H]
\begin{center}
\begin{tabular}{cc}
\includegraphics[scale=0.25]{noisysignal2.eps} & \includegraphics[scale=0.25]{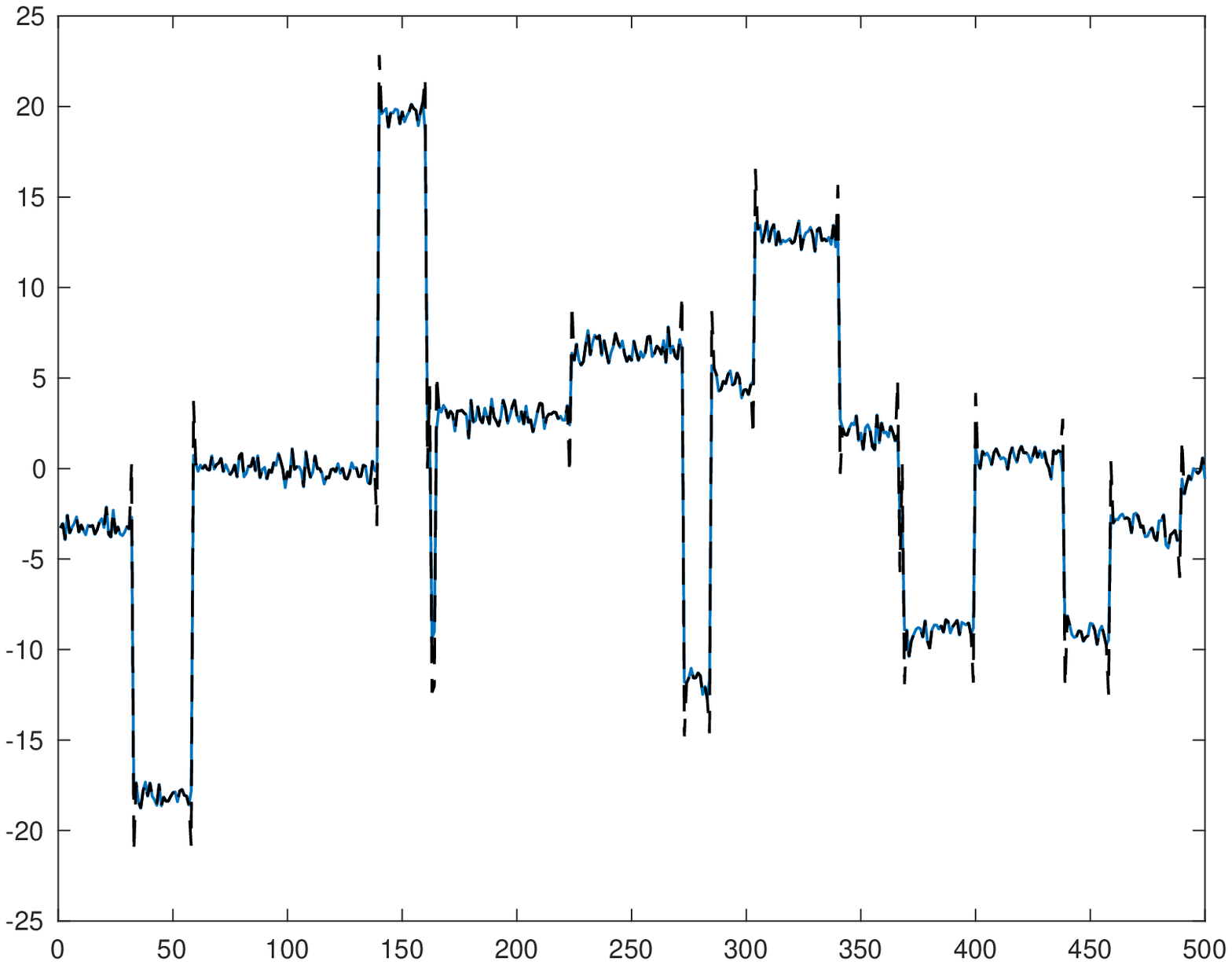} \\
(a) & (b)
\end{tabular}
\caption{Algorithm \ref{alg:DCA}. The noisy signal (a) and the L1 initialization points $z + \lambda B^\top \env_\alpha \phee(B\x{k})$ for (b)  $k = 4$. } \label{fig:DCupdates}
\end{center}
\end{figure}

Algorithm \ref{alg:HPDG} uses the proximity operator of $\phee_\alpha$ directly, splitting the problem into a sparsity update and a fidelity update. As sparsity promoting functions, the proximity operators of both $\phee$ and $\phee_\alpha$ send small entries to zero. However, the nonconvexity of $\phee_\alpha$ gives us a greater tolerance for large entries. For instance, when $\phee = \|\cdot\|_1$, $\prox_{\phee}$ shrinks all entries towards zero, while $\prox_{\phee_\alpha}$ is the identity on entries beyond a certain threshold. These large entries correspond to true signal information. This is illustrated in Figure \ref{fig:PDHGupdates}.

\begin{figure}[h]
\begin{center}
\begin{tabular}{ccc}
\includegraphics[scale=0.25]{originaldifference2.eps} &
\includegraphics[scale=0.25]{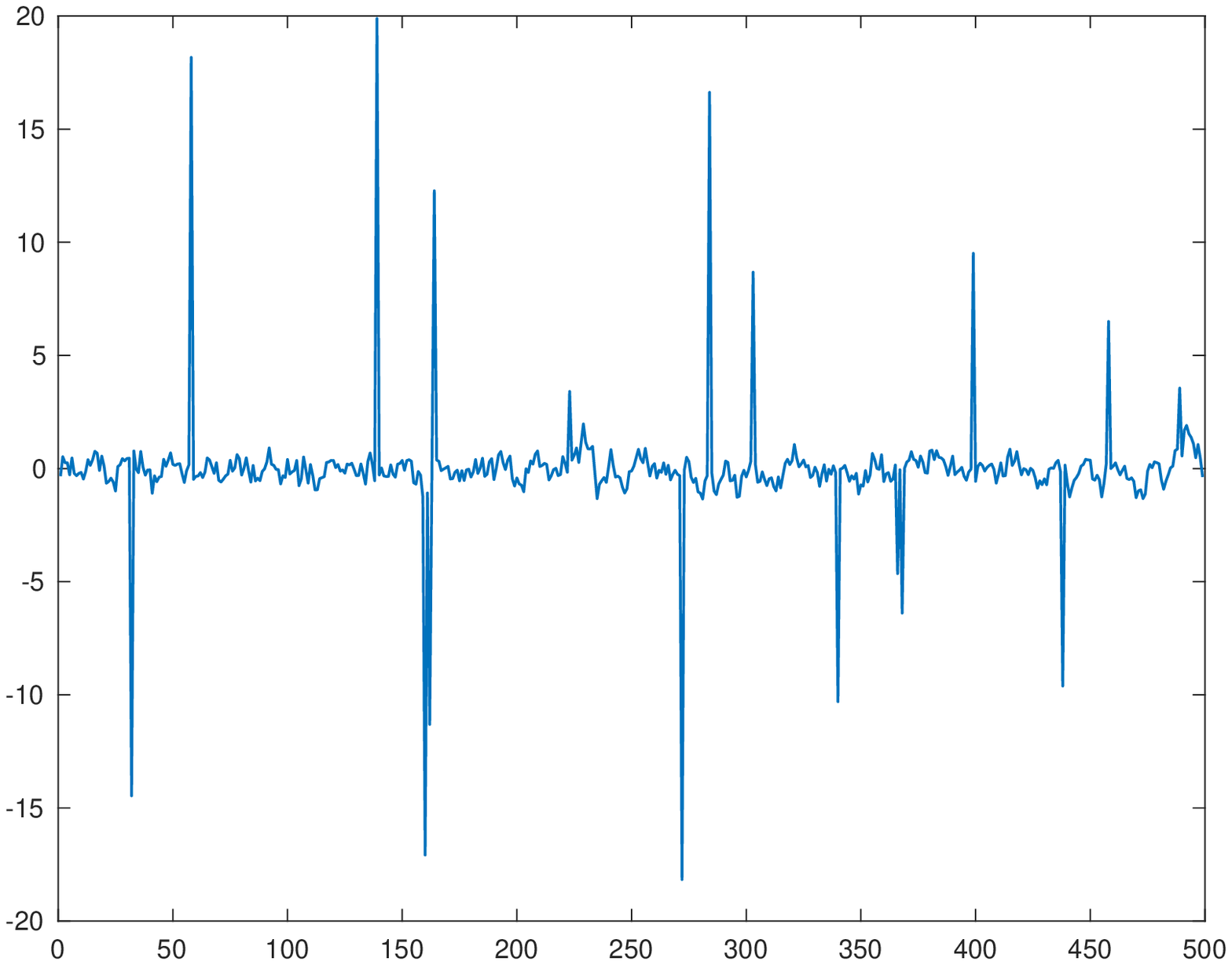} & \includegraphics[scale=0.25]{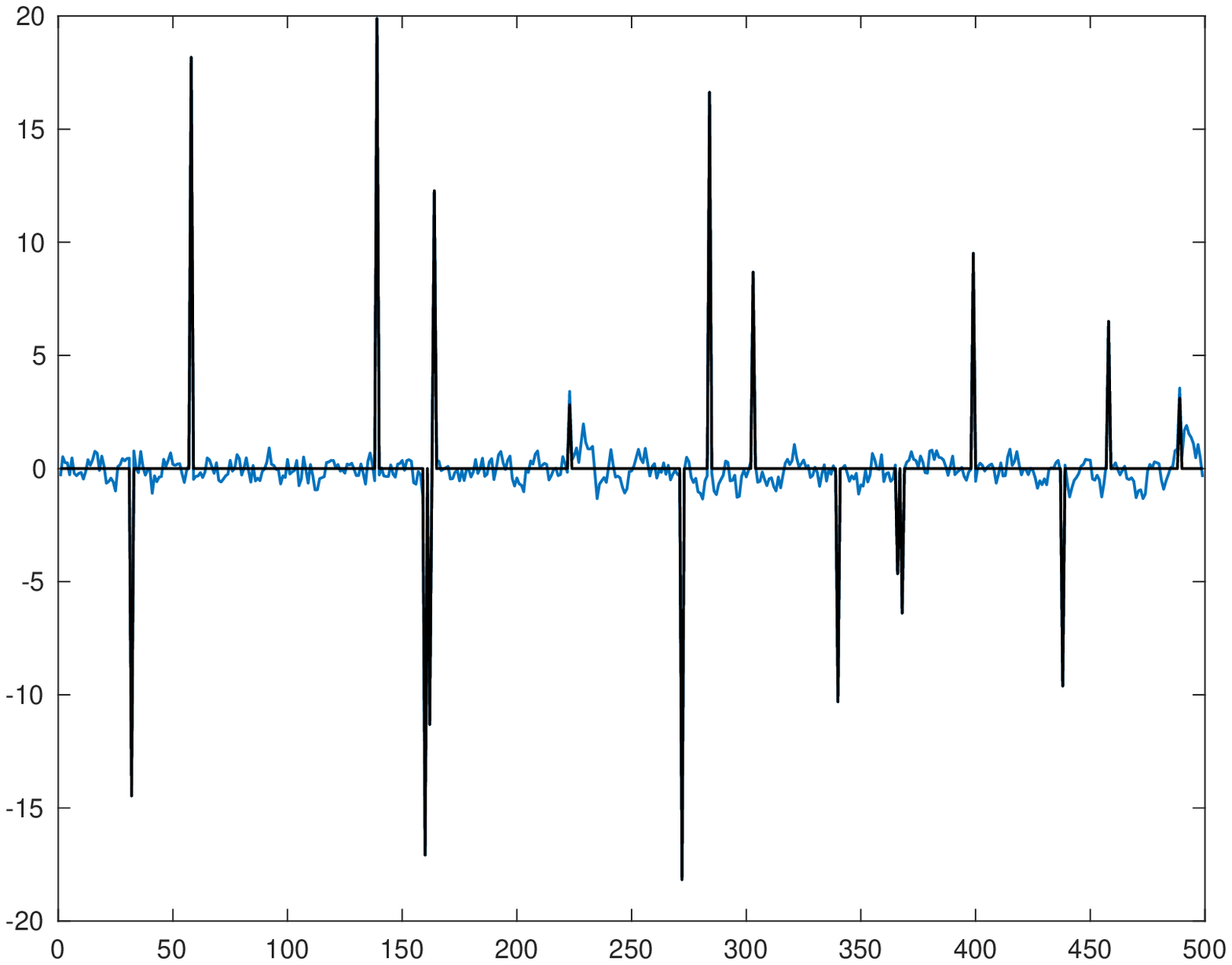} \\
(a) & (b) & (c)
\end{tabular}
\caption{Algorithm \ref{alg:HPDG}. (a) The true sparse representation of the signal, (b) $B\x{k} + \frac{1}{\sigma}\theta^{(k)}$ for $k = 10$, and (c) the update $\prox_{\sigma^{-1} \phee_\alpha}(B\x{k} + \frac{1}{\sigma}\theta^{(k)})$.}\label{fig:PDHGupdates}
\end{center}
\end{figure}

In summary, each algorithm reduces bias differently: Algorithm \ref{alg:PD} emphasizes the signal features of each primal iterate, Algorithm \ref{alg:DCA} consists of Bregman-like iterations which incorporate the boosting term into the noisy signal, and Algorithm \ref{alg:HPDG} uses the form of $\prox_{\phee_\alpha}$ directly. However, in each case we see that the inclusion of the envelope works to preserve signal features, either directly (as in the first two algorithms) or implicitly (as in the last algorithm).

\section{Numerical Experiments}\label{sec:experiments}

In this section, we specify the matrix $B$, the function $\varphi$, and the set $C$ in model~\eqref{eq:Primal} so that the resulting model is suitable for image denoising.

We choose the matrix $B$ of size $2N^2 \times N^2$ through an $N \times N$ matrix $D$ as follows:
$$
B := \begin{bmatrix}\id_N \otimes D \\  D\otimes \id_N\end{bmatrix} \quad \mbox{with} \quad D:=\begin{bmatrix}
0&&&\\
-1&1&&\\
&\ddots&\ddots&\\
&&-1&1
\end{bmatrix},
$$
where $\id_N$ is the $N \times N$ identity matrix and the notation $P \otimes Q$ denotes the Kronecker product of matrices $P$ and $Q$. We know that $\|B\|^2= 8 \sin^2 \frac{(N-1)\pi}{2N} <1$ (see, e.g., \cite{Micchelli-Shen-Xu:IP-11}).

Let $u$ be a vector in $\mathbb{R}^{2N^2}$. We choose $\varphi: \mathbb{R}^{2N^2} \rightarrow \mathbb{R}$ as a compositional norm given in \eqref{def:exampe2} with $\omega_j=\{j, N^2+j\}$, that is,
\begin{equation*}\label{eq:way2:varphi-B}
\varphi(u):=\sum_{j=1}^{N^2} \left\|\begin{bmatrix} u_j \\ u_{N^2+j} \end{bmatrix}\right\|, \quad u \in \mathbb{R}^{2N^2}.
\end{equation*}

With $B$ and $\varphi$ given in the above, $\varphi(Bx)$ is called the total variation of the image $x$ in $\mathbb{R}^{N^2}$, and the pair of $\varphi(Bx)$ with indices in $\omega_j$ is essentially the discrete gradient of the image at the $j$-th pixel. Here, $x$ is the vectorization of an image formed by stacking the columns of this image into a single column vector. For easier reading without causing ambiguity, an image is treated as a two-dimensional array and a one-dimensional vector interchangeably. Finally, since all pixel values of a gray-scale image are in $[0, 255]$, we choose $C:=[0, 255]^{N^2}$ for images in $\mathbb{R}^{N^2}$.

Prior to applying Algorithm~\ref{alg:PD} (PD), Algorithm~\ref{alg:DCA} (DCA), and Algorithm~\ref{alg:HPDG} (PDHG) for model~\eqref{eq:Primal}, we also need to know the proximity operators of the functions $\varphi$ and $\iota_C$. The proximity operator of $\varphi_\alpha$ is given in \eqref{prox-falpha-ex2}.  From the Moreau identify and \eqref{prox-ex2}, we know that for any $u \in \mathbb{R}^{2N^2}$
$$
\prox_{\sigma \varphi^*}(u)=u-\sigma \prox_{\sigma^{-1} \varphi}(\sigma u)=\sum_{j=1}^{N^2} I^\top_{w_j} \cdot \mathrm{proj}_{[0,1]}(\|I_{\omega_j}u\|) \cdot \frac{I_{\omega_j}u}{\|I_{\omega_j}u\|},
$$
which does not depend on $\sigma$.
This formula says that for each pair of $u$ with indices $\omega_j$, its projection onto the unit ball centered at the origin is the pair of $\prox_{\sigma \varphi^*}(u)$ with the same indices. For the indicator function $\iota_C$, $\prox_{\iota_C}=\mathrm{proj}_C$ which will send the values in a vector larger than 255 or lower than 0 to 255 and 0, respectively.

%

For comparison, we include the ROF model which is a special case of model~\eqref{eq:PrimalDual_Generic_Model} with $F=\frac{1}{2\lambda}\|\cdot-z\|^2$, $G=\iota_C$, and $H=\varphi$. This model is solved by the iterative scheme given in \eqref{alg:PD1}-\eqref{alg:PD3}.  The corresponding algorithm is referred to as ROF-TV algorithm.

In the rest of this section, we present all parameters used in Algorithms ROF-TV, PD, DCA, and PDHG, and compare their numerical performance for image denoising.

\subsection{Parameters and Stopping Criterion}
We first talk about the parameters related to the underlying models, then discuss the parameters associated with each algorithm, and finally describe the stopping criterion for all algorithms.

Model~\eqref{eq:Primal} involves two parameters $\lambda$ and $\alpha$. It is well known that the regularization parameter $\lambda$ varies according to the noise level of the noisy image to be denoised. From Proposition~\ref{prop:Lips}, we know that model~\eqref{eq:Primal} is strictly convex when $\alpha>\lambda \|B\|^2$. Therefore, in our experiments, we always choose $\alpha=1.5\lambda \|B\|^2$ for each given $\lambda$.

Methods of ROF-TV, PD, and DCA all exploit the iterative scheme \eqref{alg:PD1}-\eqref{alg:PD3} for which the proper values of the parameters $\sigma$, $\tau$, and $\rho$ are to be assigned. We use the model for TV algorithm as example to show how set these parameters. We reformulate the associated model~\eqref{eq:PrimalDual_Generic_Model} with $F=\frac{1}{2\lambda}\|\cdot-z\|^2$, $G=\iota_C$, and $H=\varphi$ without changing its minimizer, to the one with $F=\frac{1}{2}\|\cdot-z\|^2$, $G=\iota_C$, and $H=\lambda\varphi$. In our simulations, we choose $\sigma=0.1$, $\tau=0.99/(0.5+\sigma\|B\|^2)$, and $\rho=1$. With these chosen parameters, the sequences of $\{\x{k}\}_{k\in \mathbb{N}}$, generated ROF-TV, PD, and DCA, converge to the solutions of the corresponding optimization models, respectively.

For PDHG, we choose $\sigma=2/\alpha$, $\tau=0.99/(\sigma \|B\|^2)$, and $\rho=1$. Then, the convergence of the sequence of $\{\x{k}\}_{k\in \mathbb{N}}$,  generated by PDHG, is the consequence of Theorem~\ref{thm:HPDG}.

Iterations in the algorithms of ROF-TV, PD, DCA, and PDHG are terminated whenever the one of the following two conditions occurs: the maximum number of iterations has been exceeded or
$$
\|\x{k+1}-\x{k}\|/\|\x{k}\| \leq \texttt{tol},
$$
where $\texttt{tol}$ denotes a prescribed tolerance value. In our experiments, we set $\texttt{tol} =10^{-4}$. For  Algorithms TV, PD, and PDHG, the maximum number of iterations is set to be 300. For Algorithm DCA, there are basically two levels of looping: outer loop and inner loop. The outer loop refers to the procedure of generating $\y{k}$ and $\x{k+1}$ via \eqref{eq:DC-1} and \eqref{eq:DC-2}, respectively. The inner loop is used to find $\x{k+1}$ via an iterative scheme. We set the maximum number of iterations for the outer loop to be 10, and 100 for the inner loop.

\subsection{Numerical Results for Denoising}
In our experiments, we choose the images of ``Cameraman'' (Figure~\ref{fig:camera-visualquality}(a)),  ``House'' (Figure~\ref{fig:house-visualquality}(a)), and ``Peppers'' (Figure~\ref{fig:peppers-visualquality}(a)) with size $256 \times 256$,  as the original images $x$.  The noisy images (for example, see, Figures~\ref{fig:camera-visualquality}(b), \ref{fig:house-visualquality}(b), and \ref{fig:peppers-visualquality}(b)) are modeled as
$$
z = x+\epsilon
$$
with $\epsilon$ being the white Gaussian noise of standard derivation $\eta$. The noise at level $\eta$ being 15, 20, and 25 will be added to the test images to evaluate the performance of the proposed model and the corresponding algorithms. The quality of the denoised image $\widetilde{x}$ obtained from a denoising algorithm is measured by the peak-signal-to-noise ratio (PSNR)
$$
\mathrm{PSNR}:=20 \log_{10} \left(\frac{255}{256\|x-\widetilde{x}\|}\right).
$$

In Table~\ref{Table:cameraman-noise-lambda}, we reported the average PSNR values of the denoised images of ``Cameraman'' and the CPU time consumed by all tested algorithms for various values of $\lambda$ over 20 realizations at the same noise level. Note that algorithms PD, DCA, and PDHG are developed to find a solution to model~\eqref{eq:Primal}. From this table, we observed that PDHG performs always better than PD and DCA in terms of both the PSNR values of the denoised image and the CPU time used. The same conclusion can be drawn for the image of ``House'' as shown in Table~\ref{Table:house-noise-lambda}. Numerical results for the image of ``Peppers'' are listed in Table~\ref{Table:peppers-noise-lambda}. In this case, DCA produced better denoised images than PD and DCA in terms of the PSNR values, however, using much more CPU times. From the PSNR values in these tables, we can see that the quality of the denoised images via the optimization model penalized by the proposed structured promoting functions (solved by PD, DCA, and PDHG) is better than that with the classical ROF total variation model. For noise at level $\eta=20$, Figure~\ref{fig:psnr-cpu-camera}(a) illustrates the PSNR values of the denoised ``Cameraman'' images via all methods over 20 noise realizations while Figure~\ref{fig:psnr-cpu-camera}(b) presents the used CPU times. We can see that PDHG consistently produces the highest quality images with the least CPU time used.

Figure~\ref{fig:camera-visualquality} shows the denoised images when all algorithms apply to the noisy image of ``Cameraman'' with noise level of $20$. For the same noise level, Figure~\ref{fig:house-visualquality} shows the denoised images of ``House'' while Figure~\ref{fig:peppers-visualquality} shows the denoised images of ``Peppers''. Although all denoised images look similar, visually, we can see that the denoised images by PDHG have less artifacts than the others.

\begin{table}[H]
\caption{Numerical results of  TV, PD, DCA, and PDHG methods for the image of ``Cameraman''. The pair $(\cdot, \cdot)$ is used to report
both the PSNR value (the first number) of a denoised image and the CPU time (the second number). }
\begin{center}\small
\begin{tabular}{c|cccc}\hline
$\lambda$ &ROF&PD&DCA&PDHG \\ \hline
\multicolumn{5}{c} {White Gaussian noise with standard deviation 15}\\ \hline
 9&(30.32, 0.19)&(30.20, 0.18)&(30.17, 1.41)&(30.22, 0.18)\\
10&(30.30, 0.20)&(30.50, 0.21)&(30.44, 1.51)&(30.52, 0.18)\\
11&(30.18, 0.22)&(30.62, 0.22)&(30.54, 1.52)&(\textbf{30.65}, 0.18)\\
12&(30.01, 0.24)&(30.61, 0.24)&(30.52, 1.59)&(30.63, 0.19)\\
13&(29.80, 0.26)&(30.50, 0.27)&(30.41, 1.54)&(30.52, 0.20)\\
\hline
\multicolumn{5}{c} {White Gaussian noise with standard deviation 20}\\ \hline
14&(28.79, 0.25)&(29.00, 0.26)&(28.92, 1.70)&(29.02, 0.23)\\
15&(28.73, 0.26)&(29.10, 0.28)&(29.02, 1.80)&(29.13, 0.22)\\
16&(28.64, 0.28)&(29.13, 0.30)&(29.03, 1.94)&(\textbf{29.16}, 0.22)\\
17&(28.52, 0.30)&(29.09, 0.31)&(28.99, 2.15)&(29.11, 0.22)\\
18&(28.38, 0.33)&(29.00, 0.34)&(28.91, 1.91)&(29.03, 0.23)\\
\hline
\multicolumn{5}{c} {White Gaussian noise with standard deviation 25}\\ \hline
18&(27.67, 0.38)&(27.87, 0.41)&(27.78, 2.55)&(27.89, 0.38)\\
19&(27.65, 0.38)&(27.97, 0.39)&(27.87, 3.09)&(28.04, 0.27)\\
20&(27.60, 0.33)&(28.01, 0.36)&(27.90, 2.28)&(\textbf{28.04}, 0.26)\\
21&(27.43, 0.38)&(27.96, 0.39)&(27.85, 2.47)&(27.99, 0.26)\\
22&(27.33, 0,40)&(27.89, 0.41)&(27.78, 2.34)&(27.91, 0.26)\\
 \hline
\end{tabular}
\end{center}
\label{Table:cameraman-noise-lambda}
\end{table}

\begin{table}[H]
\caption{Numerical results of  TV, PD, DCA, and PDHG methods for the image of ``House''. The pair $(\cdot, \cdot)$ is used to report
both the PSNR value (the first number) of a denoised image and the CPU time (the second number). }
\begin{center}\small
\begin{tabular}{c|cccc}\hline
$\lambda$ &ROF&PD&DCA&PDHG \\ \hline
\multicolumn{5}{c} {White Gaussian noise with standard deviation 15}\\ \hline
 9&(32.05, 0.18)&(31.32, 0.18)&(31.30, 1.26)&(31.45, 0.15)\\
10&(32.30, 0.21)&(31.90, 0.21)&(31.30, 1.40)&(31.93, 0.16)\\
11&(32.40, 0.20)&(32.26, 0.22)&(32.18, 1.25)&(32.30, 0.16)\\
12&(32.42, 0.21)&(32.46, 0.24)&(32.35, 1.32)&(32.50, 0.16)\\
13&(32.37, 0.24)&(32.53, 0.27)&(32.41, 1.41)&(\textbf{32.56}, 0.17)\\

\hline
\multicolumn{5}{c} {White Gaussian noise with standard deviation 20}\\ \hline
14&(30.94, 0.24)&(30.64, 0.26)&(30.55, 1.62)&(30.67, 0.18)\\
15&(31.07, 0.25)&(30.95, 0.27)&(30.83, 1.71)&(30.99, 0.17)\\
16&(31.13, 0.27)&(31.15, 0.29)&(31.02, 1.57)&(31.19, 0.18)\\
17&(31.14, 0.28)&(31.27, 0.31)&(31.12, 1.67)&(31.31, 0.18)\\
18&(31.11, 0.29)&(31.31, 0.33)&(31.16, 2.18)&(\textbf{31.35}, 0.19)\\
\hline
\multicolumn{5}{c} {White Gaussian noise with standard deviation 25}\\ \hline
19&(30.01, 0.36)&(29.91, 0.41)&(29.77, 2.57)&(29.94, 0.26)\\
20&(30.10, 0.41)&(30.11, 0.48)&(29.95, 2.37)&(30.15, 0.29)\\
21&(30.14, 0.37)&(30.24, 0.42)&(30.07, 2.28)&(30.29, 0.26)\\
22&(30.15, 0.36)&(30.33, 0.40)&(30.15, 2.22)&(30.37, 0.24)\\
23&(30.14, 0.36)&(30.36, 0.43)&(30.18, 2.31)&(\textbf{30.41}, 0.24)\\
 \hline
\end{tabular}
\end{center}
\label{Table:house-noise-lambda}
\end{table}

\begin{table}[H]
\caption{Numerical results of  TV, PD, DCA, and PDHG methods for the image of ``Peppers''. The pair $(\cdot, \cdot)$ is used to report
both the PSNR value (the first number) of a denoised image and the CPU time (the second number). }
\begin{center}\small
\begin{tabular}{c|cccc}\hline
$\lambda$ &ROF&PD&DCA&PDHG \\ \hline
\multicolumn{5}{c} {White Gaussian noise with standard deviation 15}\\ \hline
 9&(31.13, 0.21)&(30.48, 0.20)&(30.58, 1.45)&(30.47, 0.19)\\
10&(31.27, 0.24)&(30.91, 0.23)&(30.01, 1.60)&(30.89, 0.20)\\
11&(31.31, 0.24)&(31.18, 0.25)&(31.28, 1.45)&(31.15, 0.18)\\
12&(31.26, 0.27)&(31.29, 0.31)&(31.40, 1.62)&(31.28, 0.21)\\
13&(31.16, 0.30)&(31.32, 0.31)&(\textbf{31.43}, 0.77)&(31.30, 0.19)\\

\hline
\multicolumn{5}{c} {White Gaussian noise with standard deviation 20}\\ \hline
14&(29.27, 0.26)&(29.50, 0.27)&(29.58, 1.76)&(29.48, 0.21)\\
15&(29.81, 0.28)&(29.70, 0.31)&(29.78, 1.94)&(29.69, 0.20)\\
16&(29.80, 0.36)&(29.82, 0.38)&(29.91, 1.92)&(29.80, 0.24)\\
17&(29.75, 0.36)&(29.87, 0.40)&(\textbf{29.96}, 2.06)&(29.85, 0.24)\\
18&(29.68, 0.38)&(39.87, 0.47)&(29.96, 2.25)&(29.85, 0.25)\\
\hline
\multicolumn{5}{c} {White Gaussian noise with standard deviation 25}\\ \hline
19&(28.66, 0.32)&(28.55, 0.37)&(28.63, 2.23)&(28.54, 0.22)\\
20&(28.67, 0.36)&(28.67, 0.40)&(28.74, 2.38)&(28.65, 0.23)\\
21&(28.65, 0.38)&(28.73, 0.42)&(28.80, 2.17)&(28.71, 0.23)\\
22&(28.61, 0.39)&(28.75, 0.45)&(\textbf{28.83}, 2.34)&(28.73, 0.25)\\
23&(28.55, 0.42)&(28.74, 0.47)&(28.82, 2.42)&(28.72, 0.25)\\
 \hline
\end{tabular}
\end{center}
\label{Table:peppers-noise-lambda}
\end{table}

\begin{figure}[H]
\centering
\begin{tabular}{cc}
  \includegraphics[width=2.8in]{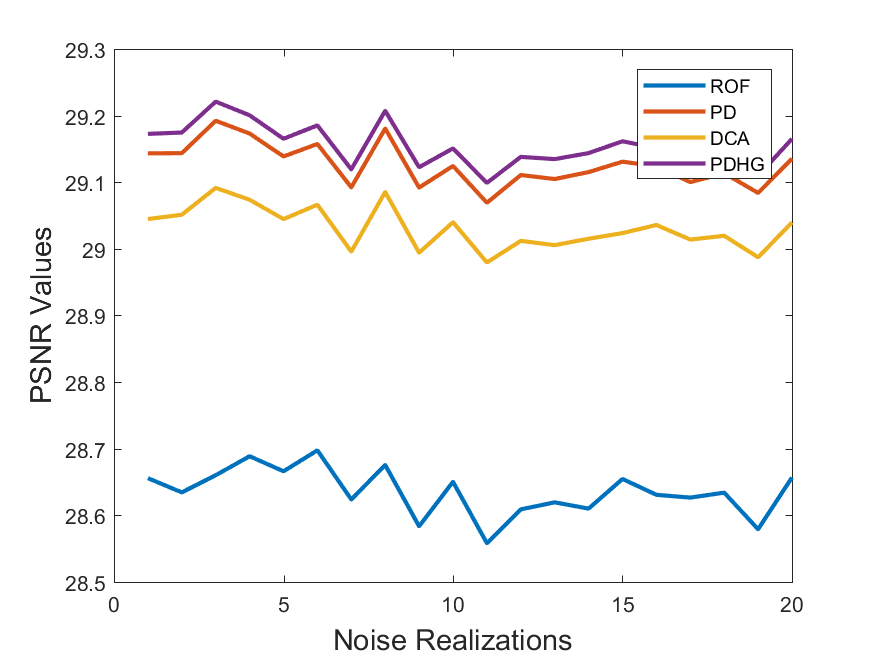}&
  \includegraphics[width=2.8in]{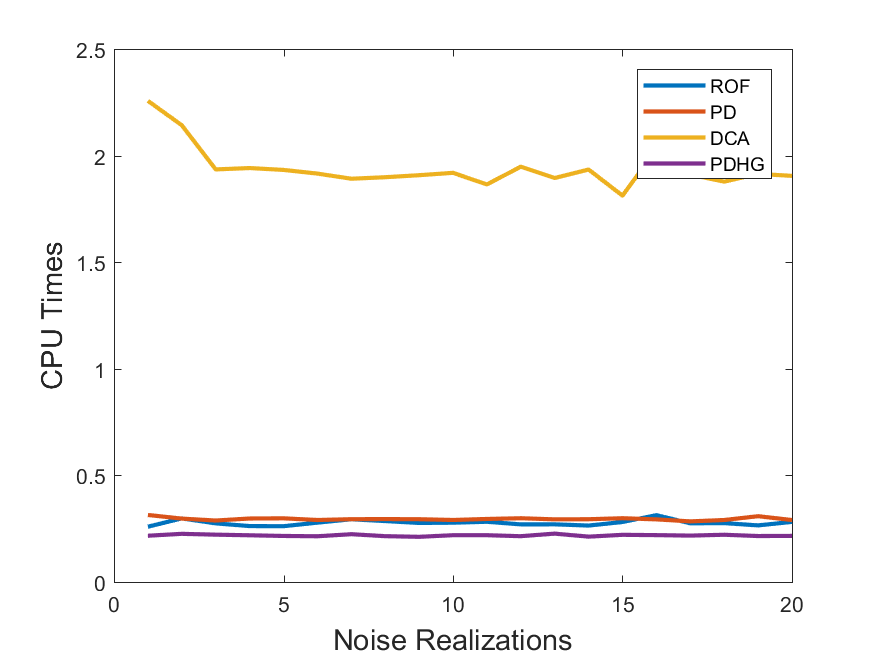}\\
  (a)&(b)
\end{tabular}
\caption{(a) The PSNR value of the denoised image of ``Cameraman'' for each Gaussian noise realization with standard deviation 20; and (b) the CPU time consumed for various algorithms.  The regularization parameter $\lambda$ is $15$ for both ROF model and \eqref{eq:Primal}.}
\label{fig:psnr-cpu-camera}
\end{figure}

\begin{figure}[H]
\centering
\begin{tabular}{ccc}
  \includegraphics[width=1.5in]{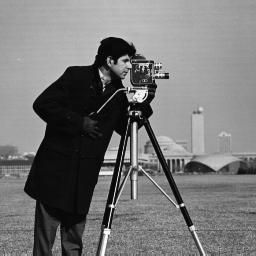}&
  \includegraphics[width=1.5in]{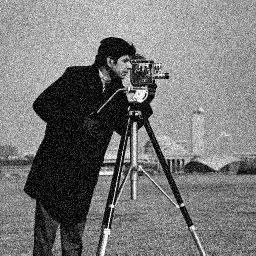}&
  \includegraphics[width=1.5in]{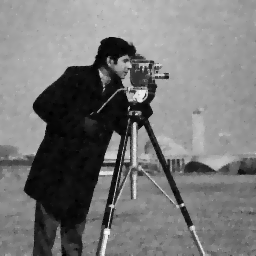}\\
  (a)&(b)&(c)\\
  \includegraphics[width=1.5in]{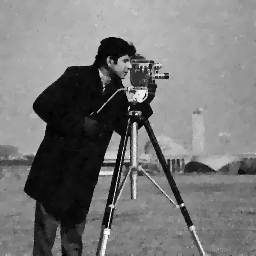}&
  \includegraphics[width=1.5in]{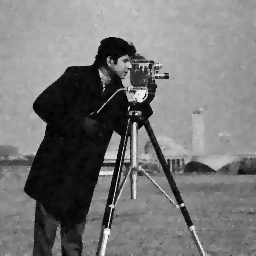}&
  \includegraphics[width=1.5in]{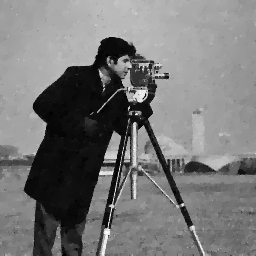}\\
  (d)&(e)&(f)
\end{tabular}
\caption{(a) The image of ``Cameraman''; (b) the image of ``Cameraman'' corrupted by Gaussian noise of standard deviation $20$; (c) the denoised image using  the ROF denoising model; the denoised images using model~\eqref{eq:Primal} by (d) PD; (e) DCA; and (f) PDHG, respectively.  The regularization parameter $\lambda$ for both models is 16.}
\label{fig:camera-visualquality}
\end{figure}

\begin{figure}[H]
\centering
\begin{tabular}{ccc}
  \includegraphics[width=1.5in]{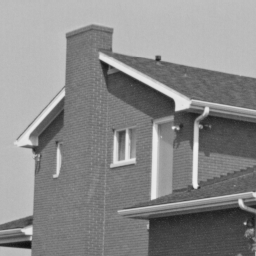}&
  \includegraphics[width=1.5in]{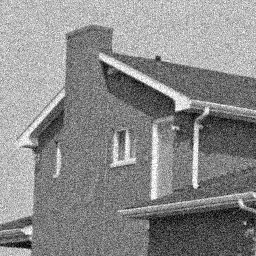}&
  \includegraphics[width=1.5in]{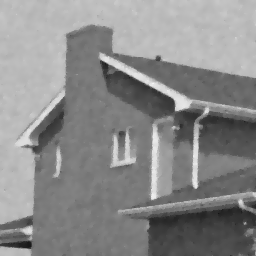}\\
  (a)&(b)&(c)\\
  \includegraphics[width=1.5in]{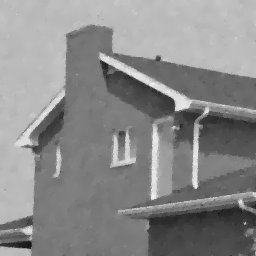}&
  \includegraphics[width=1.5in]{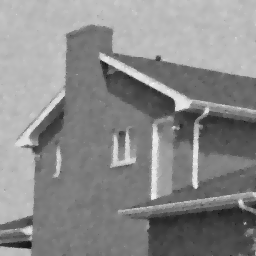}&
  \includegraphics[width=1.5in]{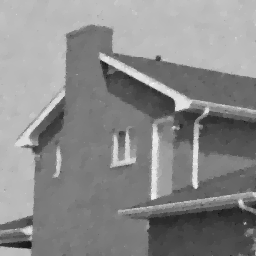}\\
  (d)&(e)&(f)
\end{tabular}
\caption{(a) The image of ``House''; (b) the image of ``House'' corrupted by Gaussian noise of standard deviation $20$; ((c) the denoised image using  the ROF denoising model; the denoised images using model~\eqref{eq:Primal} by (d) PD; (e) DCA; and (f) PDHG, respectively.  The regularization parameter $\lambda$ for both models is 18.}
\label{fig:house-visualquality}
\end{figure}

\begin{figure}[H]
\centering
\begin{tabular}{ccc}
  \includegraphics[width=1.5in]{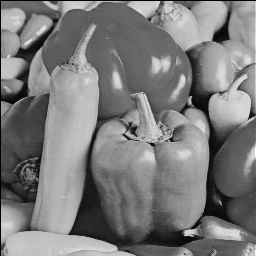}&
  \includegraphics[width=1.5in]{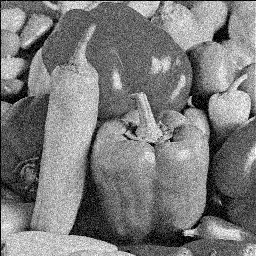}&
  \includegraphics[width=1.5in]{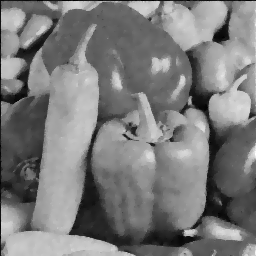}\\
  (a)&(b)&(c)\\
  \includegraphics[width=1.5in]{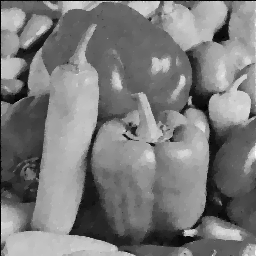}&
  \includegraphics[width=1.5in]{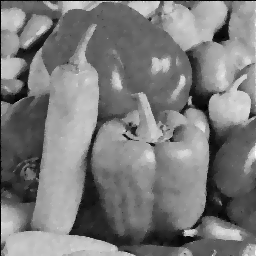}&
  \includegraphics[width=1.5in]{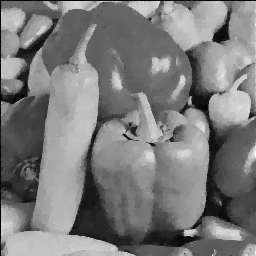}\\
  (d)&(e)&(f)
\end{tabular}
\caption{(a) The image of ``Peppers''; (b) the image of ``Peppers'' corrupted by Gaussian noise of standard deviation $20$; (c) the denoised image using  the ROF denoising model; the denoised images using model~\eqref{eq:Primal} by (d) PD; (e) DCA; and (f) PDHG, respectively.  The regularization parameter $\lambda$ for both models is 17.}
\label{fig:peppers-visualquality}
\end{figure}

\section{Concluding Remarks}\label{sec:conclusions}

We propose a general denoising model based on structured SPFs, as introduced in \cite{Shen-Suter-Tripp:2019}, and discuss various algorithms for this model. The development of these algorithms is motivated by the intrinsic structure of the model which makes it quite flexible and allows us to easily determine the convergence of the proposed methods. We illustrate the effectiveness of the proposed model by applying the modified ROF-TV model to the problem of image denoising. We see that in comparison to the traditional ROF-TV model, we are able to achieve greater accuracy without increased computation time in most cases.

Future work will feature variations of this denoising model; in particular, we are interested in the addition of a blurring kernel and applications to compressed sensing. Moreover, we believe that the structure of our proposed SPF's can be used to improve convergence results for nonconvex algorithms. Semiconvexity (or, more generally, prox-regularity) has been leveraged in this way here and elsewhere (e.g. \cite{Deng-Yin:JSC:2016}, \cite{Mollenhoff:Strekalovskiy:Moeller:Cremers:SIAMIS:2015}), but there are many other properties of these functions which may be useful.


\section*{Disclaimer and Acknowledgment of Support}
Any opinions, findings and conclusions or recommendations expressed in this material are
those of the authors and do not necessarily reflect the views of AFRL (Air Force Research
Laboratory). The work of L. Shen was supported in part
by the National Science Foundation under grant DMS-1913039.

\bibliographystyle{siam}

\end{document}